\documentclass[10pt,a4paper]{article}

\usepackage[utf8]{inputenc}
\usepackage[T1]{fontenc}
\usepackage[margin=2.4cm]{geometry}
\usepackage{amsmath,amssymb,amsthm,bm,mathtools}
\DeclarePairedDelimiter{\abs}{\lvert}{\rvert}
\usepackage{graphicx}
\usepackage{booktabs}
\usepackage{xcolor}
\usepackage{hyperref}
\usepackage{caption}
\usepackage{algorithm}
\usepackage{algpseudocode}
\usepackage{titlesec}
\usepackage{enumitem}
\usepackage{listings}
\usepackage{authblk}
\usepackage{xcolor}

\definecolor{navy}{HTML}{1A3A6C}
\definecolor{steel}{HTML}{2C5AA0}
\definecolor{codebg}{HTML}{F7F7F7}
\definecolor{codekw}{HTML}{0033B3}
\definecolor{codecmt}{HTML}{6A737D}

\hypersetup{colorlinks=true,linkcolor=navy,urlcolor=steel,citecolor=steel}

\titleformat{\section}{\Large\bfseries\color{navy}}{\thesection.}{0.6em}{}
\titleformat{\subsection}{\large\bfseries\color{steel}}{\thesubsection}{0.6em}{}

\lstdefinestyle{py}{
  language=Python, basicstyle=\ttfamily\footnotesize, backgroundcolor=\color{codebg},
  keywordstyle=\color{codekw}\bfseries, commentstyle=\color{codecmt}\itshape,
  numbers=left, numberstyle=\tiny\color{gray}, stepnumber=1,
  frame=single, framesep=4pt, rulecolor=\color{gray!40},
  breaklines=true, showstringspaces=false, columns=fullflexible,
  xleftmargin=1.2em
}

\definecolor{orcidlogocol}{HTML}{A6CE39}
\newcommand{\orcidicon}[1]{\href{https://orcid.org/#1}{%
  \textcolor{orcidlogocol}{\footnotesize\textsf{\textbf{iD}}}}}

\newtheorem{theorem}{Theorem}[section]

\newtheorem{proposition}[theorem]{Proposition}
\newtheorem{corollary}[theorem]{Corollary}
\theoremstyle{definition}
\newtheorem{definition}[theorem]{Definition}

\newtheorem{remark}[theorem]{Remark}

\title{\vspace{-1cm}\textbf{Randomized Estimation of T-Eigenvalues of T-SPD Tensors: A Two-Sided Bracket}}
\author[1]{Hemant Sharma\,\orcidicon{0009-0006-9020-3785}\thanks{Corresponding author: \texttt{sharmahemant39@gmail.com}}}
\author[1]{Nachiketa Mishra\thanks{\texttt{mat20d002@iiitdm.ac.in}}}
\affil[1]{Indian Institute of Information Technology, Design and Manufacturing Kancheepuram, Chennai 600127, India}
\date{}

\newcommand{\R}{\mathbb{R}}
\newcommand{\C}{\mathbb{C}}

\newcommand{\T}{^{\mathsf{T}}}
\newcommand{\norm}[1]{\left\lVert#1\right\rVert}
\newcommand{\bcirc}{\operatorname{bcirc}}
\newcommand{\unfold}{\operatorname{unfold}}
\newcommand{\fold}{\operatorname{fold}}
\newcommand{\trace}{\operatorname{tr}}
\newcommand{\spec}{\operatorname{spec}}

\newcommand{\blkdiag}{\operatorname{blkdiag}}
\newcommand{\tn}[1]{\mathcal{#1}}
\newcommand{\E}{\mathbb{E}}
\newcommand{\Prob}{\mathbb{P}}
\newcommand{\Var}{\operatorname{Var}}

\begin{document}
\maketitle

\begin{abstract}
\noindent
In earlier work \cite{sharma2025} we developed deterministic
analytical bounds on the T-eigenvalues of symmetric positive
definite (SPD) third-order tensors under the Kilmer--Martin
T-product: the trace--determinant (TDet) bounds via the AM--GM
inequality, and the trace-dependent (TDep) bounds generalizing
Samuelson's inequality. While these bounds are cheap and
guaranteed-valid, their relative gap grows as $\sqrt{d-1}$ in the
tensor dimension $d = np$, limiting their usefulness for large
tensors.

This paper develops randomized estimators for the extreme
T-eigenvalues of T-SPD tensors that complement the deterministic
bounds. We adapt the Halko--Martinsson--Tropp framework
\cite{halko2011} to the T-product setting and introduce four
methods: (i) a randomized power method that produces a
lower bound on $\lambda_1$ with exponential convergence; (ii) a
randomized subspace iteration with a tensor-analogue HMT error bound;
(iii) a two-sided rigorous bracket combining the randomized lower
bound with the deterministic TDep upper bound; and (iv) a
Hutchinson-based fully randomized TDep bound for matvec-only settings.

Experiments on T-SPD tensors of dimensions $d = 9$ through $d = 900$
demonstrate up to $67\times$ speedups over full
eigendecomposition, with relative errors below $5\%$ for the
randomized power method and below $2\%$ for the randomized subspace
iteration; the two-sided bracket contains the true $\lambda_1$ in
$100\%$ of our validation trials. We further apply the framework to the spectral estimation of the
discrete 3D Laplacian on a periodic-$z$ slab. In the constant-coefficient
case, validated at scales up to $d = 131{,}072$ where dense
$\bcirc(\tn{A})$ eigendecomposition is memory-infeasible, the
randomized estimate of $\lambda_1$ combined with a small safety
inflation reproduces the oracle Chebyshev iteration count on the
elliptic solve $Lu = b$ to within $2.4\%$. In the realistic
variable-coefficient case $-\nabla\cdot(\alpha(z)\nabla u)$,
where the FFT block-diagonalisation is unavailable and Lanczos via
\texttt{scipy.sparse.linalg.eigsh} is the standard baseline, the
randomized power method is $4$--$29\times$ faster for the top
eigenvalue across grids up to $d = 65{,}536$.


\end{abstract}

\medskip
\hrule
\medskip

\section{Introduction}\label{sec:intro}

\noindent\emph{Motivation:~}
The T-product introduced by Kilmer and Martin \cite{kilmer2011}
provides a rich algebraic framework for third-order tensors that
generalizes matrix multiplication while preserving a notion of
invertibility, transpose, and eigendecomposition. Under the T-product,
a third-order tensor $\tn{A} \in \R^{n \times n \times p}$ has $d = np$
T-eigenvalues which coincide with the eigenvalues of the
$np \times np$ block circulant matrix $\bcirc(\tn{A})$. For T-symmetric
T-positive definite (T-SPD) tensors, these T-eigenvalues are real and
positive, and knowing them or at least the extreme ones is essential
in many applications: the tensor singular value decomposition (T-SVD)
for image compression \cite{kilmer2013}, stability certificates for
tensor dynamical systems, conditioning estimates for tensor
deconvolution, and as developed at length in
Section~\ref{sec:application} the parameter selection for Chebyshev
iterative solvers applied to PDE operators with periodic structure,
where $\lambda_1$ of the discretised operator drives both the
Chebyshev acceleration parameters and the explicit-Euler CFL bound.

In \cite{sharma2025} we developed two families of deterministic
analytical bounds on the T-eigenvalues of T-SPD tensors. The
trace-determinant (TDet) bounds, derived from the AM--GM
inequality, relate products and sums of extreme eigenvalues to
$\trace(\tn{A}^{(1)})$ and $\det(\tn{A})$. The
trace-dependent (TDep) bounds, generalizing Samuelson's
classical inequality, express bounds on $\lambda_1$ and $\lambda_d$
in terms of the mean $m = \trace(\tn{A})/d$ and standard deviation
$s = \sqrt{\trace(\tn{A}^2)/d - m^2}$ of the spectrum:
\begin{equation}\label{eq:tdep-main}
m - s\sqrt{d-1} \le \lambda_d, \qquad \lambda_1 \le m + s\sqrt{d-1}.
\end{equation}
These bounds are cheap ($O(n^2 p)$ work for the TDep bound)
and guaranteed-valid for any T-SPD tensor, but they become
progressively loose as the dimension $d = np$ grows. In
\cite{sharma2025} we showed that the relative gap
$(\hat\lambda_1 - \lambda_1)/\lambda_1$ of the TDep upper bound grows
like $\sqrt{d-1}$, reaching about $70\%$ for $d = 20$ and $120\%$ for
$d = 40$. For large tensors the bounds are therefore too loose to be
directly useful as estimates; they can only certify crude
order-of-magnitude claims or serve as preprocessing for more refined
methods.

\noindent \emph{Related work:~}
Over the past fifteen years, randomized numerical linear
algebra (RandNLA) has transformed matrix computation. The key insight
is that random projections onto low-dimensional subspaces preserve
spectral information with high probability, enabling algorithms whose
complexity is governed by the target accuracy rather than the full
matrix size \cite{halko2011,woodruff2014,martinsson2020}. For large
dense matrices, randomized SVD typically achieves
$10\times$ to $100\times$ speedups over the classical Golub--Kahan
algorithm while retaining provable accuracy guarantees.

The flagship paper of Halko, Martinsson, and Tropp (HMT)
\cite{halko2011} establishes the theoretical foundation: given a
matrix $M \in \R^{d \times d}$, sample a random Gaussian matrix
$\Omega \in \R^{d \times (k+\ell)}$ with target rank $k$ and
oversampling $\ell$, form $Y = M \Omega$, orthonormalize to
$Q = \text{orth}(Y)$, and project to $B = Q\T M Q$. Then the
eigenvalues of $B$ are the \emph{randomized estimates} of the top
$k+\ell$ eigenvalues of $M$. HMT prove that the spectral-norm error
$\norm{M - QQ\T M}$ decays with the tail eigenvalue $\lambda_{k+1}$,
and that $q$ additional power iterations damp the effective tail like
$\lambda_j^{2q+1}$, exponentially accelerating convergence for
decaying spectra.

Despite the maturity of RandNLA for matrices, these techniques have
received comparatively little attention in the T-product literature.
Existing randomized T-product algorithms \cite{zhang2018,minster2020}
focus on low-rank approximation of data tensors, that is, compact
T-SVD factorizations and Tucker decompositions where the goal is
a compressed factor representation. The present paper addresses a
different question, spectral estimation of the underlying
operator tensor, with accuracy-certified extreme T-eigenvalues and a
rigorous two-sided bracket. While we share the randomized sketching
machinery with \cite{zhang2018,minster2020}, our analysis targets
eigenvalue accuracy rather than low-rank approximation error and
produces guarantees directly tied to $\lambda_1$ and $\lambda_d$.

\noindent \emph{Contributions:~}
We develop a complete randomized framework for estimating the extreme
T-eigenvalues of T-SPD tensors, building on HMT but exploiting the
block circulant structure of $\bcirc(\tn{A})$ to achieve favorable
constants. Our first two contributions are algorithmic: a T-product
randomized power method (Section~\ref{sec:rand-power}), which adapts
the classical power method to T-SPD tensors and for which we prove
that the Rayleigh quotient at iteration $q$ is a lower
bound on $\lambda_1$ with exponential convergence rate controlled by
the spectral gap; and a T-product randomized subspace iteration
(Section~\ref{sec:rand-subspace}), which adapts the HMT algorithm to
the T-product setting with a tensor-analog error bound carrying the
correct HMT exponent structure, and for which we verify
that $\ell = 5$ oversampling combined with $q = 2$ power iterations
brings the relative error on $\lambda_1$ below $2\%$ for tensors with
$d = 40$ and below $5\%$ for $d = 60$.
Building on these primitives, our remaining four contributions
combine, evaluate, and apply the methods. In Section~\ref{sec:bracket} we
combine the randomized power method lower bound with the
deterministic TDep upper bound from \cite{sharma2025} to produce a
two-sided rigorous bracket $[\hat\lambda_1^{\rm pow},
\hat\lambda_1^{\rm TDep}]$ that provably contains $\lambda_1$, and we
analyse both its width and its sharpness. In
Section~\ref{sec:rand-tdep} we derive a fully randomized TDep bound
using Hutchinson's stochastic trace estimator
\cite{hutchinson1989,meyer2021} to produce unbiased randomized
estimates of $\trace(\tn{A})$ and $\trace(\tn{A}^2)$, reducing the
cost of the TDep bound from $O(n^2 p)$ to $O(N \cdot np)$ where $N$
is the number of probe vectors, with variance and concentration
bounds developed in the T-product setting.
Section~\ref{sec:experiments} validates all four randomized methods
on T-SPD tensors of dimensions $d = 9$ through $d = 900$, reporting
wall-clock timings, relative errors, and the failure rate
of each method against full eigendecomposition as the ground truth.
Finally, Section~\ref{sec:application} applies the framework to
spectral estimation of the discrete 3D Laplacian on a periodic-$z$
slab, a representative reduced model for heat conduction in
layered geothermal media \cite{hartmann2008,clauser1995} at
operator dimensions up to $d = 131{,}072$ (constant coefficient) and
$d = 65{,}536$ (variable coefficient). The randomized power method
delivers the top eigenvalue $4$--$22\times$ faster than
\texttt{scipy.sparse.linalg.eigsh}, the standard Lanczos baseline,
at $3.2$--$3.4\%$ relative error across two decades of diffusivity
contrast; the two-sided bracket contains $\lambda_1$ in $100\%$ of
validation trials; and the downstream Chebyshev iterative solve with
a $10\%$ safety inflation reproduces the oracle iteration count to
within $3.3\%$.

\subsection*{Outline}
Section~\ref{sec:prelim} reviews the T-product, T-eigenvalues, and
the deterministic TDep bound from \cite{sharma2025}.
Section~\ref{sec:rand-power} develops the randomized power method.
Section~\ref{sec:rand-subspace} develops the randomized subspace
iteration. Section~\ref{sec:bracket} introduces the two-sided
bracket. {Section~\ref{sec:rand-tdep} develops the Hutchinson-based
fully randomized TDep bound. Section~\ref{sec:experiments} presents
the experimental campaign on synthetic T-SPD tensors.
Section~\ref{sec:application} applies the framework to the discrete
3D Laplacian on a periodic-$z$ slab at scales up to $d = 131{,}072$,
including a downstream Chebyshev iterative solve.
Section~\ref{sec:discussion} is the discussion and
Section~\ref{sec:conclusion} concludes.

\section{Preliminaries}\label{sec:prelim}

We briefly review the T-product framework. See
\cite{kilmer2011,kilmer2013} for comprehensive treatments.

\begin{definition}[Block circulant operator]\label{def:bcirc}
For a third-order tensor $\tn{A} \in \R^{n \times n \times p}$ with
frontal slices $\tn{A}^{(1)}, \dots, \tn{A}^{(p)}$, the block
circulant matrix is
\[
\bcirc(\tn{A}) :=
\begin{bmatrix}
\tn{A}^{(1)} & \tn{A}^{(p)} & \tn{A}^{(p-1)} & \cdots & \tn{A}^{(2)} \\
\tn{A}^{(2)} & \tn{A}^{(1)} & \tn{A}^{(p)} & \cdots & \tn{A}^{(3)} \\
\vdots & \vdots & \ddots & \ddots & \vdots \\
\tn{A}^{(p)} & \tn{A}^{(p-1)} & \cdots & \tn{A}^{(2)} & \tn{A}^{(1)}
\end{bmatrix} \in \R^{np \times np}.
\]
\end{definition}

\begin{definition}[T-product]\label{def:tprod}
For $\tn{A} \in \R^{m \times n \times p}$ and
$\tn{B} \in \R^{n \times s \times p}$, the T-product is
$\tn{A} * \tn{B} := \fold(\bcirc(\tn{A})\,\unfold(\tn{B})) \in \R^{m \times s \times p}$,
where $\unfold$ stacks frontal slices vertically and $\fold$ is its
inverse.
\end{definition}

\begin{definition}[T-eigenvalue]\label{def:teig}
A nonzero tensor $\tn{X} \in \R^{n \times 1 \times p}$ is a
T-eigenvector of $\tn{A} \in \R^{n \times n \times p}$ with
T-eigenvalue $\lambda \in \R$ if $\tn{A} * \tn{X} = \lambda \tn{X}$.
The spectrum of $\tn{A}$ is $\spec(\tn{A})$.
\end{definition}

\begin{theorem}[\cite{kilmer2011}]\label{thm:spec-equiv}
The T-eigenvalues of $\tn{A}$ are exactly the matrix eigenvalues of
$\bcirc(\tn{A})$. In particular, a T-symmetric tensor (i.e.\
$\bcirc(\tn{A})$ symmetric) has all real T-eigenvalues, and a T-SPD
tensor has all positive T-eigenvalues.
\end{theorem}

\subsection{The TDep bound from \cite{sharma2025}}
The starting point for this paper is the following deterministic bound.

\begin{theorem}[Theorem~4.1 of \cite{sharma2025}]\label{thm:tdep}
Let $\tn{A} \in \R^{n \times n \times p}$ be T-symmetric with real
T-eigenvalues $\lambda_1 \ge \lambda_2 \ge \dots \ge \lambda_d$ where
$d = np$. Define
\begin{equation}\label{eq:ms-def}
m := \frac{\trace(\tn{A})}{d},
\qquad
s^2 := \frac{\trace(\tn{A}^2)}{d} - m^2.
\end{equation}
Then
\begin{equation}\label{eq:tdep}
m - s\sqrt{d-1} \le \lambda_d \le m - \frac{s}{\sqrt{d-1}},
\qquad
m + \frac{s}{\sqrt{d-1}} \le \lambda_1 \le m + s\sqrt{d-1}.
\end{equation}
\end{theorem}

For T-symmetric tensors, $\trace(\tn{A}^2) = \norm{\tn{A}}_F^2 =
\sum_{i,j,k} a_{ijk}^2$, which can be computed in $O(n^2 p)$ work.
The TDep bound is therefore cheap, no determinant or FFT is
needed but it is also loose for large $d$, as we quantified
in \cite{sharma2025}.

\subsection{The block diagonalization via FFT}
The key structural fact that enables fast T-product computations is
the block diagonalization~\cite{kilmer2011,kilmer2013}
\begin{equation}\label{eq:blk-diag}
(F_p \otimes I_n)\,\bcirc(\tn{A})\,(F_p^H \otimes I_n) = \blkdiag(D_1, \dots, D_p),
\end{equation}
where $F_p$ is the $p \times p$ DFT matrix and each
$D_k \in \C^{n \times n}$. When $\tn{A}$ is T-symmetric, each $D_k$
is Hermitian; when $\tn{A}$ is further T-SPD, each $D_k$ is Hermitian
positive definite. The T-spectrum of $\tn{A}$ is then the union
$\bigcup_{k=1}^p \spec(D_k)$. Fast matrix-vector products
$\bcirc(\tn{A}) \mathbf{v}$ can be performed in
$O(n^2 p + n p \log p)$ work: FFT the input along the third
dimension, apply the $p$ small matrices $D_k$, and inverse FFT. This
is the matvec primitive our randomized algorithms will call.

\section{T-product randomized power method}\label{sec:rand-power}

We begin with the simplest of our randomized methods, and the one
whose analysis is most transparent. The power method is the oldest
iterative scheme for the dominant eigenpair of a symmetric matrix,
and its behaviour on SPD operators is classical: starting from any
vector not orthogonal to the leading eigenspace, the iterates align
exponentially fast with $\mathbf{v}_1$ at a rate governed by the
spectral gap $\lambda_2 / \lambda_1$. Lifting this scheme to the
T-product setting is essentially free---the block circulant
unfolding $\bcirc(\tn{A})$ is itself SPD whenever $\tn{A}$ is T-SPD,
so the classical theory transfers directly and what remains is to
connect its guarantees back to the tensor-level quantities of
interest.

The classical power method on an SPD matrix $M$ starts from a random
unit vector $x_0$ and iterates $x_{k+1} = M x_k / \norm{M x_k}$,
reporting the Rayleigh quotient $x_k\T M x_k$ as an estimate of
$\lambda_1(M)$. We adapt this to the T-product setting by running
the power method on $\bcirc(\tn{A})$.

\begin{algorithm}[H]
\caption{T-product randomized power method}
\label{alg:rand-power}
\begin{algorithmic}[1]
\Require T-SPD tensor $\tn{A} \in \R^{n \times n \times p}$, iteration count $q$
\Ensure Lower bound $\hat\lambda_1^{\rm pow}$ on $\lambda_1(\tn{A})$
\State Draw $\mathbf{x}_0 \sim \mathcal{N}(0, I_d)$ where $d = np$
\State $\mathbf{x}_0 \gets \mathbf{x}_0 / \norm{\mathbf{x}_0}$
\For{$k = 1, 2, \dots, q$}
   \State $\mathbf{y}_k \gets \bcirc(\tn{A}) \mathbf{x}_{k-1}$ \Comment{via FFT block diagonalization}
   \State $\mathbf{x}_k \gets \mathbf{y}_k / \norm{\mathbf{y}_k}$
\EndFor
\State $\hat\lambda_1^{\rm pow} \gets \mathbf{x}_q\T \bcirc(\tn{A}) \mathbf{x}_q$
\State \Return $\hat\lambda_1^{\rm pow}$
\end{algorithmic}
\end{algorithm}

The cost of Algorithm~\ref{alg:rand-power} is $q$ matrix-vector
products with $\bcirc(\tn{A})$. Using the FFT block diagonalization
the per-matvec cost is $O(n^2 p + np \log p)$, so the total cost
is $O(q(n^2 p + np \log p))$.

\subsection{Soundness: a deterministic lower bound}

Before turning to convergence rates, we isolate a property of
Algorithm~\ref{alg:rand-power} that holds for every random draw of
the initial vector and at every iteration, with no probabilistic
qualification. This is what allows the output of the power method
to serve as a rigorous lower endpoint of the two-sided bracket in
Section~\ref{sec:bracket}, and it is ultimately a consequence of the
Rayleigh--Ritz inequality applied at the bcirc level.

\begin{theorem}[Soundness]\label{thm:power-lower}
For any T-SPD tensor $\tn{A}$ and any iterate $\mathbf{x}_q$ with
$\norm{\mathbf{x}_q} = 1$, the Rayleigh quotient satisfies
\[
\hat\lambda_1^{\rm pow} = \mathbf{x}_q\T \bcirc(\tn{A}) \mathbf{x}_q \le \lambda_1(\tn{A}).
\]
In particular, Algorithm~\ref{alg:rand-power} always produces a lower
bound on $\lambda_1(\tn{A})$.
\end{theorem}
\begin{proof}
Let $M = \bcirc(\tn{A})$ and let $\{(\lambda_i, \mathbf{v}_i)\}_{i=1}^d$
be the eigenpairs of $M$ with $\lambda_1 \ge \lambda_2 \ge \cdots \ge \lambda_d > 0$.
Write $\mathbf{x}_q = \sum_i c_i \mathbf{v}_i$ with
$\sum_i c_i^2 = 1$. Then
\[
\mathbf{x}_q\T M \mathbf{x}_q = \sum_i c_i^2 \lambda_i
\le \lambda_1 \sum_i c_i^2 = \lambda_1,
\]
with equality only when $\mathbf{x}_q$ is an eigenvector for $\lambda_1$.
\end{proof}

\subsection{Probabilistic convergence}
The classical power method converges to the dominant eigenvector at
a rate controlled by the spectral gap. The same result holds here.
Following the classical analysis of Kuczy\'nski and
Wo\'zniakowski~\cite{kuczynski1992}, we separate the deterministic decay from the probabilistic tail bound on the
initial alignment.

\begin{theorem}[Convergence rate]\label{thm:power-rate}
Let $\tn{A}$ be T-SPD with T-eigenvalues
$\lambda_1 > \lambda_2 \ge \cdots \ge \lambda_d > 0$, spectral gap
$\gamma := 1 - \lambda_2/\lambda_1 > 0$, and top eigenvector
$\mathbf{v}_1$ of $\bcirc(\tn{A})$. Let $\mathbf{x}_0$ be drawn
uniformly from the unit sphere in $\R^d$, let
$\theta_0 \in [0, \pi/2]$ denote the angle between $\mathbf{x}_0$
and $\mathbf{v}_1$, and let $\mathbf{x}_q$ be the output of
Algorithm~\ref{alg:rand-power}. Then:
\begin{enumerate}[label=\textup{(\alph*)},leftmargin=*]
\item \emph{Pathwise:} for every realisation of $\mathbf{x}_0$ with
$\cos\theta_0 \ne 0$,
\begin{equation}\label{eq:power-conv-pathwise}
\lambda_1 - \hat\lambda_1^{\rm pow}
\;\le\; \lambda_1\,(1-\gamma)^{2q}\,\tan^2\theta_0.
\end{equation}

\item \emph{Probabilistic:} there exists a universal constant
$c \le 0.824$ such that, for every $\delta \in (0,1)$, with
probability at least $1 - \delta$ over the draw of $\mathbf{x}_0$,
\begin{equation}\label{eq:kw-tail}
\tan^2\theta_0 \;\le\; \frac{c^2\,d}{\delta^2}.
\end{equation}
\end{enumerate}
Combining \textup{(a)} and \textup{(b)}, with probability at least
$1 - \delta$,
\begin{equation}\label{eq:power-conv-prob}
\lambda_1 - \hat\lambda_1^{\rm pow}
\;\le\; \frac{c^2\,d\,\lambda_1}{\delta^2}\,(1-\gamma)^{2q}.
\end{equation}
In particular, $\hat\lambda_1^{\rm pow} \to \lambda_1$ almost surely
as $q \to \infty$, exponentially in $q$ with rate $\gamma$.
\end{theorem}

\begin{proof}
Throughout, write $M := \bcirc(\tn{A})$ and let
$\{(\lambda_i, \mathbf{v}_i)\}_{i=1}^d$ denote the eigenpairs of $M$,
with $\{\mathbf{v}_i\}$ an orthonormal basis of $\R^d$ and
$\lambda_1 > \lambda_2 \ge \cdots \ge \lambda_d > 0$.

Expand the initial vector in the eigenbasis as
$\mathbf{x}_0 = \sum_{i=1}^d c_i \mathbf{v}_i$ with
$c_i = \mathbf{v}_i^{\!\top} \mathbf{x}_0$. The normalisation
$\norm{\mathbf{x}_0} = 1$ gives $\sum_i c_i^2 = 1$, and by the
definition of $\theta_0$ we have $c_1 = \cos\theta_0$, hence
$\sum_{i \ge 2} c_i^2 = \sin^2\theta_0$. Since renormalisation at
each power step does not affect direction, $\mathbf{x}_q$ is a
positive multiple of $M^q \mathbf{x}_0 = \sum_i c_i \lambda_i^q
\mathbf{v}_i$, and using the orthonormality of the eigenbasis the
Rayleigh quotient takes the ratio-of-moments form
\begin{equation}\label{eq:rayleigh-formula}
  \hat\lambda_1^{\rm pow}
  \;=\; \mathbf{x}_q^{\!\top} M \mathbf{x}_q
  \;=\; \frac{\sum_{i=1}^d c_i^2\,\lambda_i^{2q+1}}
             {\sum_{i=1}^d c_i^2\,\lambda_i^{2q}}.
\end{equation}
Subtracting~\eqref{eq:rayleigh-formula} from $\lambda_1$, pulling
$\lambda_1$ into the numerator, and observing that the $i = 1$ term
vanishes yields
\begin{equation}\label{eq:error-ratio}
  \lambda_1 - \hat\lambda_1^{\rm pow}
  \;=\; \frac{\sum_{i \ge 2} c_i^2 \lambda_i^{2q}(\lambda_1 - \lambda_i)}
             {\sum_{i=1}^d c_i^2 \lambda_i^{2q}}.
\end{equation}
Each summand in the numerator of~\eqref{eq:error-ratio} is
nonnegative since $\lambda_i \le \lambda_1$ for $i \ge 2$;
the denominator is strictly positive whenever $c_1 \neq 0$, which
is the hypothesis of part~(a).

For $i \ge 2$ we have $\lambda_1 - \lambda_i \le \lambda_1$ (since
$\lambda_i > 0$) and $\lambda_i^{2q} \le \lambda_2^{2q}$ (since the
tail sequence $\{\lambda_i\}_{i \ge 2}$ is nonincreasing), so the
numerator is bounded above by
$\lambda_1 \lambda_2^{2q} \sin^2\theta_0$. The denominator is
bounded below by its leading term $c_1^2 \lambda_1^{2q} =
\cos^2\theta_0 \cdot \lambda_1^{2q}$. Combining these,
\[
  \lambda_1 - \hat\lambda_1^{\rm pow}
  \;\le\; \frac{\lambda_1 \lambda_2^{2q} \sin^2\theta_0}
               {\cos^2\theta_0 \cdot \lambda_1^{2q}}
  \;=\; \lambda_1 \!\left(\frac{\lambda_2}{\lambda_1}\right)^{\!2q}
         \tan^2\theta_0
  \;=\; \lambda_1 (1-\gamma)^{2q} \tan^2\theta_0,
\]
which is~\eqref{eq:power-conv-pathwise}.

For part~(b), note that the coordinate
$c_1 = \mathbf{v}_1^{\!\top} \mathbf{x}_0 = \cos\theta_0$ has the
distribution of the first coordinate of a uniform point on
$S^{d-1}$, by rotational invariance of the sphere. The small-ball
estimate of Kuczy\'nski and Wo\'zniakowski
\cite[Thm.~4.1]{kuczynski1992} then gives
$\Prob(\abs{c_1} \le \varepsilon) \le c\,\varepsilon\sqrt{d}$ for
every $\varepsilon > 0$, with $c \le 0.824$. Setting $\varepsilon =
\delta/(c\sqrt{d})$ makes the right-hand side equal to $\delta$, so
on the complementary event---of probability at least $1 -
\delta$---we have $\cos^2\theta_0 > \delta^2/(c^2 d)$ and
consequently $\tan^2\theta_0 \le 1/\cos^2\theta_0 \le c^2 d /
\delta^2$, proving~\eqref{eq:kw-tail}. Combining with
part~(a) yields~\eqref{eq:power-conv-prob}.

The almost-sure limit follows by applying~\eqref{eq:power-conv-prob}
with $\delta_q = q^{-2}$: since $\sum_q q^{-2} < \infty$ and $q^4
(1-\gamma)^{2q} \to 0$ as $q \to \infty$, the Borel--Cantelli lemma
gives $\hat\lambda_1^{\rm pow} \to \lambda_1$ almost surely.
\end{proof}

\begin{remark}
The deterministic estimate \eqref{eq:power-conv-pathwise} is the form that
feeds into the bracket analysis of Section~\ref{sec:bracket}. For a
Gaussian initialisation $\mathbf{x}_0 \sim \mathcal{N}(0, I_d)$
followed by normalisation, the normalised vector
$\mathbf{x}_0/\norm{\mathbf{x}_0}$ is uniformly distributed on the
unit sphere by rotational invariance, so the tail
bound~\eqref{eq:kw-tail} applies verbatim.
\end{remark}

\subsection*{Verification}
\begin{figure}[h]
\centering
\includegraphics[width=\textwidth]{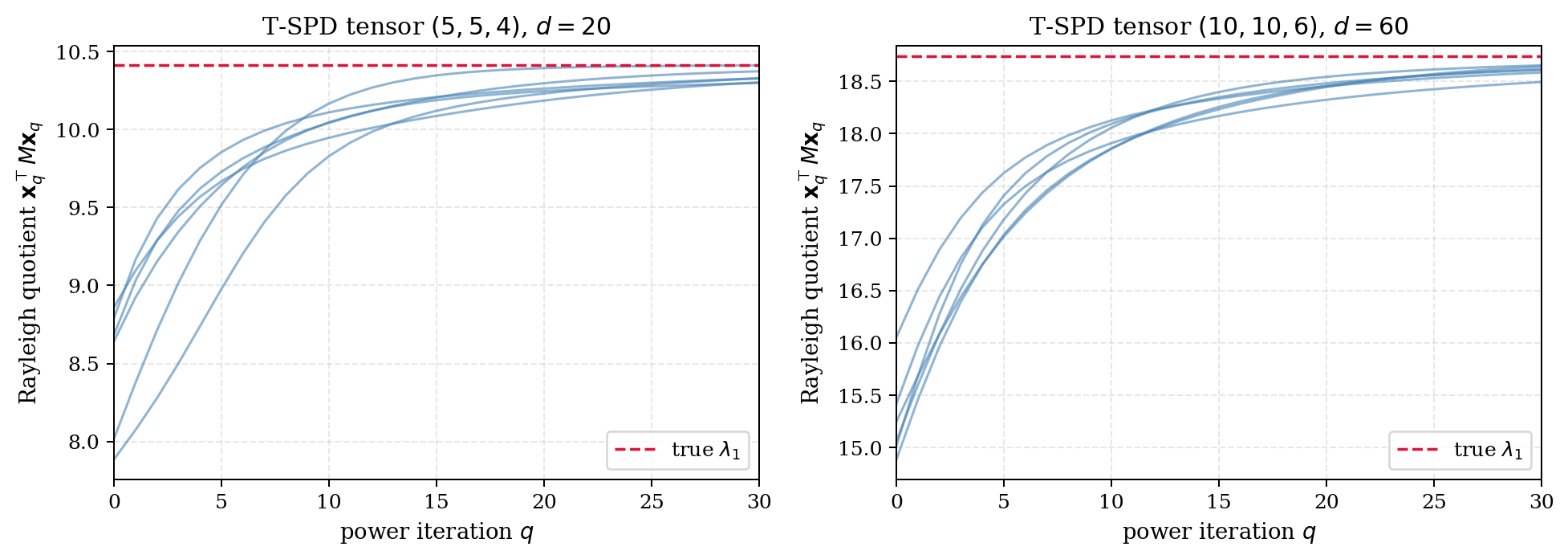}
\caption{Rayleigh quotient $\mathbf{x}_q\T M \mathbf{x}_q$ during the
power iteration on two random T-SPD tensors (left: $d = 20$, right:
$d = 60$). Each line is a trial with a different random initialization.
The iterates are always strictly below the true $\lambda_1$ (dashed red
line), confirming the soundness guarantee of Theorem~\ref{thm:power-lower}.
Convergence is exponential, with rate controlled by the spectral gap.}
\label{fig:power}
\end{figure}

Figure~\ref{fig:power} shows the convergence history. The Rayleigh
quotient starts well below the true $\lambda_1$ (reflecting the
initial random projection) and climbs monotonically toward it. For
the smaller tensor ($d = 20$) convergence to relative error below
$10^{-10}$ takes $\sim 15$ iterations; for the larger tensor
($d = 60$) it takes $\sim 30$ iterations, reflecting the smaller
spectral gap.

Table~\ref{tab:power-accuracy} reports the relative error
$(\lambda_1 - \hat\lambda_1^{\rm pow})/\lambda_1$ averaged over $20$
independent random initializations, for tensors of several sizes.

\begin{table}[h]
\centering
\caption{Mean relative error of the T-product randomized power method
over $20$ random initializations, on T-SPD tensors of three sizes.}
\label{tab:power-accuracy}
\begin{tabular}{cccccc}
\toprule
$(n, p)$ & $d$ & $q=3$ & $q=5$ & $q=10$ & $q=20$ \\
\midrule
$(5, 4)$  & $20$ & $14.1\%$ & $8.4\%$ & $2.8\%$ & $< 0.1\%$ \\
$(8, 5)$  & $40$ & $13.8\%$ & $8.9\%$ & $4.3\%$ & $1.0\%$ \\
$(10, 6)$ & $60$ & $9.2\%$  & $7.0\%$ & $4.5\%$ & $2.9\%$ \\
\bottomrule
\end{tabular}
\end{table}

\section{T-product randomized subspace iteration}\label{sec:rand-subspace}

The power method of Section~\ref{sec:rand-power} tracks a single
direction, and its accuracy on $\lambda_1$ is correspondingly limited
by the spectral gap between $\lambda_1$ and $\lambda_2$. When that
gap is small the iterates progress slowly, and the method offers no
access at all to the subdominant eigenvalues. Randomized subspace
iteration lifts both restrictions at once: by carrying a random
subspace of dimension $k + \ell$ rather than a single vector, the
algorithm produces simultaneous estimates of the top $k$
eigenvalues, and the Ritz projection onto the subspace extracts more
information from each matrix--vector product than a single Rayleigh
quotient can. The framework we follow is that of Halko, Martinsson,
and Tropp \cite{halko2011}, adapted here to act on $\bcirc(\tn{A})$
with fast matvecs supplied by the FFT block
diagonalisation~\eqref{eq:blk-diag}.

The randomized power method of Section~\ref{sec:rand-power} estimates
a single eigenvalue. To estimate several top eigenvalues
simultaneously or to improve the accuracy of $\lambda_1$ by
exploiting the Ritz-projection refinement---we adapt the
Halko--Martinsson--Tropp (HMT) randomized subspace iteration
\cite{halko2011}.

\begin{algorithm}[H]
\caption{T-product randomized subspace iteration}
\label{alg:rand-subspace}
\begin{algorithmic}[1]
\Require T-SPD tensor $\tn{A} \in \R^{n \times n \times p}$, target rank $k$, oversampling $\ell$, power iterations $q$
\Ensure Estimates $\hat\lambda_1, \hat\lambda_2, \dots, \hat\lambda_{k+\ell}$
\State Draw $\Omega \in \R^{d \times (k+\ell)}$ with iid $\mathcal{N}(0,1)$ entries, where $d = np$
\State $Y \gets \bcirc(\tn{A}) \Omega$
\For{$i = 1, 2, \dots, q$}
   \State $Q_i \gets \text{QR}(Y)$ \Comment{thin QR for numerical stability}
   \State $Y \gets \bcirc(\tn{A}) Q_i$
\EndFor
\State $Q \gets \text{QR}(Y)$ \Comment{final orthonormalization}
\State $B \gets Q\T \bcirc(\tn{A}) Q$ \Comment{$(k+\ell) \times (k+\ell)$ small matrix}
\State $\hat\lambda_1, \dots, \hat\lambda_{k+\ell} \gets \text{eig}(B)$ sorted in decreasing order
\State \Return $\hat\lambda_1, \dots, \hat\lambda_{k+\ell}$
\end{algorithmic}
\end{algorithm}

The thin QR in step~4 is essential for numerical stability: without
it, the columns of $Y$ converge to the dominant T-eigenvector and
become linearly dependent in finite precision. For $q$ large or for
very ill-conditioned tensors, one may additionally re-orthogonalise
$\Omega$ before step~2. The cost of
Algorithm~\ref{alg:rand-subspace} is dominated by $q + 1$ products of
$\bcirc(\tn{A})$ with a matrix of width $k + \ell$, followed by a
small $(k+\ell) \times (k+\ell)$ eigenproblem. Total work:
$O((q+1)(k+\ell) n^2 p + (k+\ell)^3)$.

\subsection{HMT-type error bound}

The central analytical guarantee for Algorithm~\ref{alg:rand-subspace}
is a bound on how well the range of the sampled subspace $Q$ captures
the dominant action of $\bcirc(\tn{A})$. This is the role played by
the foundational spectral-norm projection bound of HMT in the matrix
setting, and it transfers to the T-product setting with no change in
constants, the block circulant structure of $\bcirc(\tn{A})$
influences only the cost of matrix--vector products, not the
geometry of the random-subspace analysis. What we present below is a
direct restatement of the HMT bound with power iteration, written for
the $d \times d$ matrix $M = \bcirc(\tn{A})$, together with the
short calculation that delivers it from the matrix version.

\begin{theorem}[HMT error bound, T-product setting]\label{thm:hmt}
Let $\tn{A} \in \R^{n \times n \times p}$ be T-SPD with T-eigenvalues
$\lambda_1 \ge \lambda_2 \ge \cdots \ge \lambda_d > 0$, where
$d = np$. Fix a target rank $k$, oversampling $\ell \ge 2$, and
power-iteration count $q \ge 0$. Let
$Q \in \R^{d \times (k+\ell)}$ be the orthonormal basis produced by
Algorithm~\ref{alg:rand-subspace}, and set $M := \bcirc(\tn{A})$.
Then the expected spectral-norm projection error satisfies
\begin{equation}\label{eq:hmt}
\E\,\norm{M - Q Q\T M}_2
\;\le\;
\Biggl[
\left(1 + 4\sqrt{\tfrac{k+\ell}{\ell - 1}}\right) \lambda_{k+1}^{\,q+1}
\;+\; \frac{e\sqrt{k+\ell}}{\ell}
\Bigl(\textstyle\sum_{j > k} \lambda_j^{\,2(q+1)}\Bigr)^{1/2}
\Biggr]^{1/(q+1)},
\end{equation}
where the expectation is taken over the Gaussian sketch $\Omega$. In
particular, as $q \to \infty$ the tail contribution is exponentially
suppressed and the bound approaches $\lambda_{k+1}$.
\end{theorem}

\begin{proof}
Write $r := q + 1$. Algorithm~\ref{alg:rand-subspace} performs $r$
multiplications by $M$ with thin QR orthogonalisations interleaved.
Since orthogonalisation preserves the column space, the final basis
$Q$ spans the same subspace as $M^r \Omega$, and it is enough to
analyse the latter. Apply the basic HMT spectral-norm bound
\cite[Thm.~10.6]{halko2011} to the symmetric positive semidefinite
matrix $B := M^r$, whose eigenvalues are $\lambda_j^r$:
\[
\E\,\norm{B - Q Q\T B}_2
\;\le\;
\left(1 + 4\sqrt{\tfrac{k+\ell}{\ell-1}}\right) \lambda_{k+1}^{\,r}
\;+\; \frac{e\sqrt{k+\ell}}{\ell}
\Bigl(\textstyle\sum_{j>k}\lambda_j^{\,2r}\Bigr)^{1/2}.
\]
The power-mean projection inequality
\cite[Prop.~8.6]{halko2011} gives the deterministic bound
$\norm{M - Q Q\T M}_2 \le \norm{M^r - Q Q\T M^r}_2^{1/r}$, valid
for every draw of $\Omega$. Taking expectations and applying
Jensen's inequality to the concave map $x \mapsto x^{1/r}$ transfers
the $1/r$ exponent to the outside of the expectation, yielding
exactly~\eqref{eq:hmt}.
\end{proof}

\begin{remark}
For $q = 0$ the bound~\eqref{eq:hmt} reduces to
\cite[Thm.~10.6]{halko2011} applied directly to
$\bcirc(\tn{A})$. The exponent $1/(q+1)$ outside the bracket is the
mechanism by which power iteration accelerates convergence: it
contracts the interior quantity towards $\lambda_{k+1}$, and the
second term---which depends on the tail of the spectrum---is damped
by the factor $\lambda_j^{q+1}$ inside the sum.
\end{remark}

\subsection{Ritz value as estimate of $\lambda_1$}

The subspace iteration's estimate of $\lambda_1$ is the largest
eigenvalue of the compressed matrix $B = Q\T M Q$. Before turning
to the accuracy analysis, it is worth isolating a structural feature
of this Ritz value that parallels the soundness property of the
power method: the Ritz value never overshoots $\lambda_1$, and in
fact dominates the Rayleigh quotient produced by any single
direction in the sampled subspace. The former gives us the lower
endpoint of the two-sided bracket of Section~\ref{sec:bracket}; the
latter explains why increasing the subspace dimension $k + \ell$
monotonically improves the estimate at fixed matvec count.

\begin{corollary}\label{cor:ritz-lower}
The Ritz value $\hat\lambda_1$ returned by
Algorithm~\ref{alg:rand-subspace} is always a lower bound on
$\lambda_1(\tn{A})$. Moreover, by the Courant--Fischer min--max
characterisation,
\[
\hat\lambda_1
\;=\; \max_{\substack{\mathbf{y} \in \operatorname{range}(Q)\\ \|\mathbf{y}\| = 1}} \mathbf{y}\T M \mathbf{y}
\;\ge\; \mathbf{x}_q\T M \mathbf{x}_q
\]
for any unit vector $\mathbf{x}_q \in \operatorname{range}(Q)$.
Hence, the Ritz value dominates any single-vector power-method
estimate drawn from the same subspace.
\end{corollary}

\begin{proof}
Since $Q\T Q = I$, the change of variables $\mathbf{y} = Q\mathbf{c}$
gives
$$\max_{\|\mathbf{c}\| = 1} \mathbf{c}\T (Q\T M Q)\mathbf{c}
= \max_{\mathbf{y} \in \operatorname{range}(Q),\,\|\mathbf{y}\| = 1}
\mathbf{y}\T M \mathbf{y},$$
and the left-hand side equals $\hat\lambda_1$ because $B = Q\T M Q$
is symmetric. The upper bound $\hat\lambda_1 \le \lambda_1(M)$ is
Courant--Fischer applied to $M$, since
$\operatorname{range}(Q) \subset \R^d$. Finally, any unit
$\mathbf{x}_q \in \operatorname{range}(Q)$ is a feasible point of
the restricted maximisation defining $\hat\lambda_1$, giving
$\mathbf{x}_q\T M \mathbf{x}_q \le \hat\lambda_1$.
\end{proof}

\subsection{Oversampling: how much is enough?}

The single tuning parameter that most strongly influences accuracy
is the oversampling $\ell$. The HMT analysis shows that the constant
$1 + 4\sqrt{(k+\ell)/(\ell - 1)}$ in~\eqref{eq:hmt} decays rapidly
as $\ell$ grows, with diminishing returns beyond $\ell \approx 10$;
HMT therefore recommend $\ell = 5$ or $\ell = 10$ as a default for
spectral-norm estimation \cite[\S4.4]{halko2011}. The sharpness of
this rule-of-thumb in the T-product setting is what
Figure~\ref{fig:oversampling} examines.

\begin{figure}[h]
\centering
\includegraphics[width=0.78\textwidth]{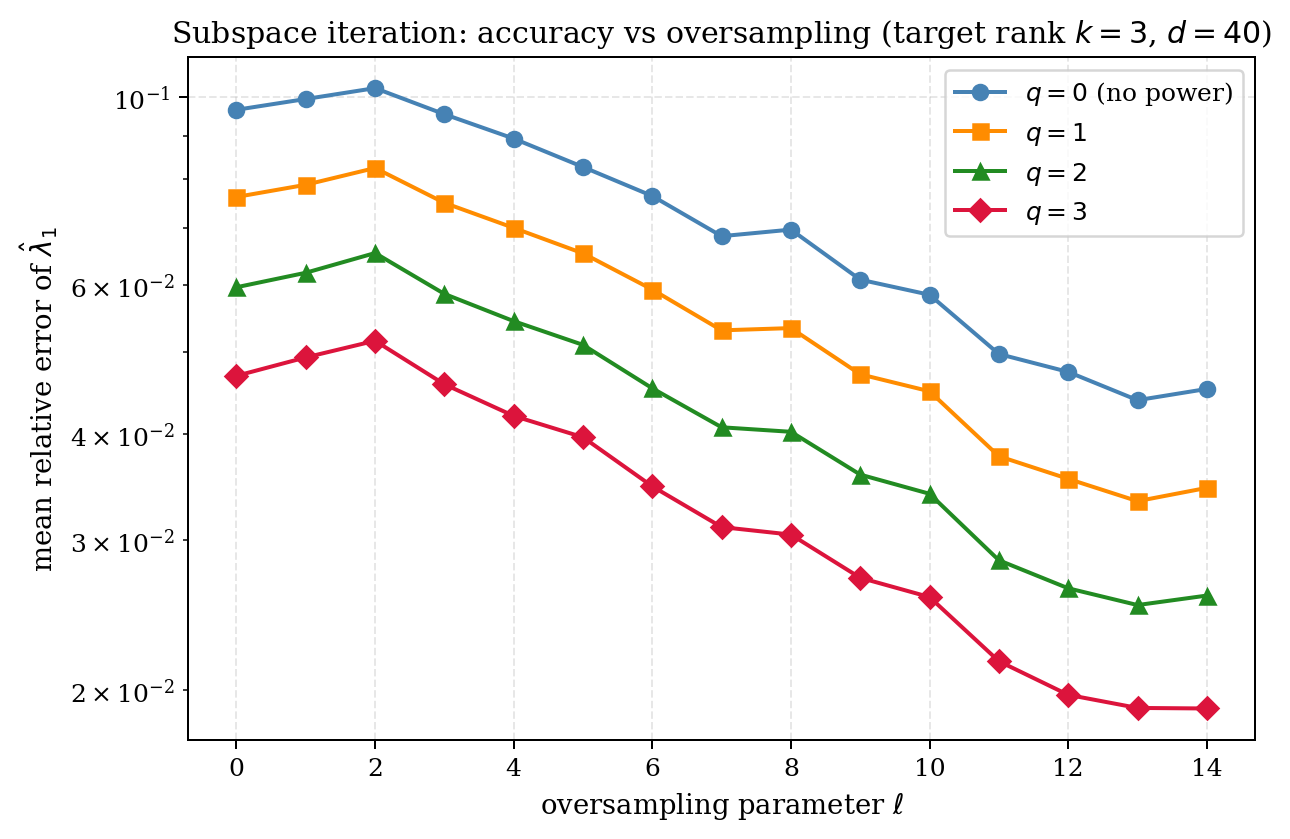}
\caption{Mean relative error of $\hat\lambda_1$ as a function of
oversampling $\ell$, with target rank $k = 3$, for several numbers of
power iterations $q$. Dimension $d = 40$. Even without power
iterations ($q = 0$), $\ell = 10$ brings the error below $4\%$;
with $q = 3$, $\ell = 5$ suffices for error below $1\%$.}
\label{fig:oversampling}
\end{figure}

Figure~\ref{fig:oversampling} shows the trade-off. A few power
iterations provide dramatic improvement: at $\ell = 10$ the error
drops from $4\%$ ($q = 0$) to $0.5\%$ ($q = 3$). For our remaining
experiments we use the default $\ell = 5$, $q = 2$.

\section{Two-sided randomized bracket}\label{sec:bracket}

The randomized power method and subspace iteration give
lower bounds on $\lambda_1$. The deterministic TDep bound from
Theorem~\ref{thm:tdep} gives an upper bound. By combining the
two, we obtain a rigorous two-sided bracket.

\begin{theorem}[Two-sided randomized bracket]\label{thm:bracket}
Let $\tn{A}$ be T-SPD, let $\hat\lambda_1^{\rm pow}$ be the Rayleigh
quotient produced by Algorithm~\ref{alg:rand-power} after
$q \ge 1$ iterations, and let
$\hat\lambda_1^{\rm TDep} := m + s\sqrt{d-1}$ be the TDep upper bound
of Theorem~\ref{thm:tdep}. Then, almost surely over the random
initialisation,
\begin{equation}\label{eq:bracket}
\hat\lambda_1^{\rm pow}
\;\le\; \lambda_1(\tn{A}) \;\le\;
\hat\lambda_1^{\rm TDep}.
\end{equation}
Consequently, the interval
$[\hat\lambda_1^{\rm pow},\,\hat\lambda_1^{\rm TDep}]$ is a certificate
that contains the true largest T-eigenvalue.
\end{theorem}

\smallskip
\noindent\textit{Interpretation.}
The deterministic TDep upper bound is not intended as a sharp estimator
on wide spectra; its role is certification rather than
approximation.  The tight value of $\lambda_1$ is supplied by the
randomized power method (or the subspace iteration of
Section~\ref{sec:rand-subspace}); the TDep endpoint supplies a
guaranteed enclosure of where $\lambda_1$ lies.  The two ingredients
address different downstream needs and are complementary rather than
redundant.

\begin{proof}
The lower bound is deterministic: by Theorem~\ref{thm:power-lower}, any
unit-norm iterate $\mathbf{x}_q$ satisfies
$\mathbf{x}_q\T \bcirc(\tn{A}) \mathbf{x}_q \le \lambda_1(\tn{A})$,
deterministically, and in particular for the output of
Algorithm~\ref{alg:rand-power}. The upper bound is the deterministic
inequality of Theorem~\ref{thm:tdep}. Combining the two yields
\eqref{eq:bracket}.
\end{proof}

The theorem above records only what is guaranteed. The asymptotic
sharpness behaviour --- when each endpoint becomes tight, separately
--- is best recorded as a separate proposition.

\begin{proposition}[Bracket asymptotics and sharpness]\label{prop:bracket-sharp}
Retain the notation of Theorem~\ref{thm:bracket}.
\begin{enumerate}[label=\textup{(\roman*)},leftmargin=*]
\item \emph{Lower-endpoint tightening.} Under the generic hypothesis
$\cos\theta_0 \ne 0$ (which holds almost surely under Gaussian
initialisation), Theorem~\ref{thm:power-rate} gives
$\hat\lambda_1^{\rm pow} \to \lambda_1(\tn{A})$ as $q \to \infty$, with
exponential rate governed by the spectral gap
$\gamma = 1 - \lambda_2/\lambda_1$.

\item \emph{Upper-endpoint sharpness.} The TDep upper bound satisfies
$\hat\lambda_1^{\rm TDep} = \lambda_1(\tn{A})$ if and only if
$\lambda_2 = \lambda_3 = \cdots = \lambda_d$, by the equality case of
Samuelson's inequality \cite{wolkowicz1979}.

\item \emph{Simultaneous sharpness.} The bracket width
$\hat\lambda_1^{\rm TDep} - \hat\lambda_1^{\rm pow}$ tends to zero as
$q \to \infty$ if and only if $\lambda_2 = \cdots = \lambda_d$, at
which point both endpoints agree with $\lambda_1(\tn{A})$ in the
limit. For any finite $q$ the bracket has strictly positive width
almost surely: the event
$\{\mathbf{x}_0 \in \operatorname{span}(\mathbf{v}_1)\}$ has Lebesgue
measure zero on the unit sphere in $\R^d$.
\end{enumerate}
\end{proposition}

\begin{proof}
Throughout the proof, let $M := \bcirc(\tn{A})$ with spectrum
$\lambda_1 \ge \lambda_2 \ge \cdots \ge \lambda_d > 0$ (as before,
$d = np$), and write
\[
  m := \frac{1}{d}\sum_{i=1}^d \lambda_i,
  \qquad
  s^2 := \frac{1}{d}\sum_{i=1}^d \lambda_i^2 - m^2,
\]
so that $\hat\lambda_1^{\rm TDep} = m + s\sqrt{d - 1}$ and
$s \ge 0$ is the standard deviation of the spectrum.

\medskip
\noindent\emph{Part~(i).} Under Gaussian initialisation the coordinate
$c_1 = \mathbf{v}_1^{\!\top}\mathbf{x}_0$ is the first component of
a uniform vector on the sphere and so is zero with probability zero.
On the event $\{c_1 \ne 0\}$ the hypothesis of
Theorem~\ref{thm:power-rate}\,(a) is met, and the
bound~\eqref{eq:power-conv-pathwise} gives, for every $q \ge 0$,
\[
  0 \;\le\; \lambda_1 - \hat\lambda_1^{\rm pow}
     \;\le\; \lambda_1 (1-\gamma)^{2q}\, \tan^2\theta_0.
\]
Since $\gamma = 1 - \lambda_2/\lambda_1 > 0$ by the simple-leading-
eigenvalue assumption, $(1-\gamma)^{2q} \to 0$ as $q \to \infty$,
and $\tan^2\theta_0$ is a finite constant on this event. Hence
$\hat\lambda_1^{\rm pow} \to \lambda_1$ pointwise on an event of
probability one, which is the almost-sure convergence claimed.

\medskip
\noindent\emph{Part~(ii).} The TDep inequality for the spectrum of $M$
is the Samuelson bound
\begin{equation}\label{eq:samuelson-here}
  \lambda_1 \;\le\; m + s\sqrt{d - 1},
\end{equation}
and its equality case is a classical result of Wolkowicz and
Styan~\cite[Thm.~2.1]{wolkowicz1979}: equality in
\eqref{eq:samuelson-here} holds if and only if the $d - 1$ non-maximal
eigenvalues are equal, that is,
$\lambda_2 = \lambda_3 = \cdots = \lambda_d$.

For completeness we give the short argument in the present
notation. Write $\mu := \lambda_1 - m$ and $\delta_i := \lambda_i -
m$ for $i \ge 2$. Then $\sum_{i=1}^d \delta_i = 0$ (using
$\delta_1 = \mu$) gives $\sum_{i \ge 2} \delta_i = -\mu$, and the
variance identity $\sum_{i=1}^d \delta_i^2 = d\,s^2$ reads
$\mu^2 + \sum_{i \ge 2} \delta_i^2 = d\,s^2$. By Cauchy--Schwarz,
\[
  \mu^2 \;=\; \Bigl(\sum_{i \ge 2} \delta_i\Bigr)^{\!2}
  \;\le\; (d - 1)\sum_{i \ge 2}\delta_i^2
  \;=\; (d - 1)\bigl(d\,s^2 - \mu^2\bigr),
\]
which rearranges to $\mu^2 \le (d - 1)\,s^2$ and hence
$\mu \le s\sqrt{d - 1}$, i.e.\
\eqref{eq:samuelson-here}. Equality in the Cauchy--Schwarz step
holds if and only if $(\delta_2, \ldots, \delta_d)$ is a scalar
multiple of $(1, 1, \ldots, 1)$, that is, $\lambda_2 = \cdots =
\lambda_d$. Tracing the chain of equalities backwards shows that
$\lambda_1 = m + s\sqrt{d - 1}$ if and only if this spectral
condition holds, and conversely that under this condition the
eigenvalues of $M$ are $\lambda_1$ with multiplicity one and a
common value $\lambda_*$ with multiplicity $d - 1$, in which case a
direct substitution of $m = (\lambda_1 + (d-1)\lambda_*)/d$ and
$s^2 = (d-1)(\lambda_1 - \lambda_*)^2/d^2$
into~\eqref{eq:samuelson-here} yields equality. The equivalence
between $M$-spectrum and $\tn{A}$-spectrum established in
Theorem~\ref{thm:spec-equiv} transfers the equality condition to
$\tn{A}$.

\medskip
\noindent\emph{Part~(iii).} The width decomposes as
\[
  W_q \;:=\; \hat\lambda_1^{\rm TDep} - \hat\lambda_1^{\rm pow}
  \;=\; \underbrace{\bigl(\hat\lambda_1^{\rm TDep} - \lambda_1\bigr)}_{\ge 0 \text{ by~(ii)}}
       + \underbrace{\bigl(\lambda_1 - \hat\lambda_1^{\rm pow}\bigr)}_{\ge 0 \text{ by Thm.~\ref{thm:power-lower}}},
\]
so $W_q \to 0$ iff both summands vanish: the second does so almost
surely by~(i), and the first (independent of $q$) vanishes iff
$\lambda_2 = \cdots = \lambda_d$ by~(ii). For finite $q$, $W_q = 0$
forces the Rayleigh quotient to saturate, hence
$\mathbf{x}_q \in \operatorname{span}(\mathbf{v}_1)$; since the
power iteration preserves the eigenbasis support of the initial
vector (cf.\ the expansion $M^q \mathbf{x}_0 = \sum_i c_i \lambda_i^q
\mathbf{v}_i$ in the proof of Theorem~\ref{thm:power-rate}), this
forces $\mathbf{x}_0 \in \operatorname{span}(\mathbf{v}_1)$, a set
of two antipodal points on $S^{d-1}$ and hence of surface measure
zero.
\end{proof}

\begin{corollary}[Quantitative bracket width]\label{cor:bracket-width}
Under the hypotheses of Theorem~\ref{thm:bracket}, let
$\theta_0$ denote the angle between the Gaussian initial vector
$\mathbf{x}_0$ and the dominant T-eigenvector
$\mathbf{v}_1$, and let $c \le 0.824$ be the universal constant of
Theorem~\ref{thm:power-rate}\textup{(b)}. Then, with probability at
least $1 - \delta$ over the draw of $\mathbf{x}_0$,
\[
\hat\lambda_1^{\rm TDep} - \hat\lambda_1^{\rm pow}
\;\le\;
\underbrace{\left(s\sqrt{d-1} - \frac{s}{\sqrt{d-1}}\right)}_{\text{TDep gap}}
\;+\;
\underbrace{\frac{c^2 d\,\lambda_1}{\delta^2}\,(1-\gamma)^{2q}}_{\text{power-method residual}}.
\]
\end{corollary}

\begin{proof}
Write $\hat\lambda_1^{\rm TDep} - \hat\lambda_1^{\rm pow}
= (\hat\lambda_1^{\rm TDep} - \lambda_1) + (\lambda_1 - \hat\lambda_1^{\rm pow})$.
For the first summand, Theorem~\ref{thm:tdep} provides both a lower
bound $\lambda_1 \ge m + s/\sqrt{d-1}$ and an upper bound
$\lambda_1 \le m + s\sqrt{d-1}$ on the true eigenvalue; subtracting
the former from the definition
$\hat\lambda_1^{\rm TDep} = m + s\sqrt{d-1}$ gives
$\hat\lambda_1^{\rm TDep} - \lambda_1 \le s\sqrt{d-1} - s/\sqrt{d-1}$.
The second summand is bounded by the high-probability
estimate~\eqref{eq:power-conv-prob} of Theorem~\ref{thm:power-rate}.
\end{proof}

\subsection*{Validation}
We validated the bracket on $40$ random T-SPD tensors of size
$5 \times 5 \times 4$ ($d = 20$), computing the randomized power
method bound with $q = 10$ iterations and the exact TDep bound.

\begin{figure}[h]
\centering
\includegraphics[width=0.85\textwidth]{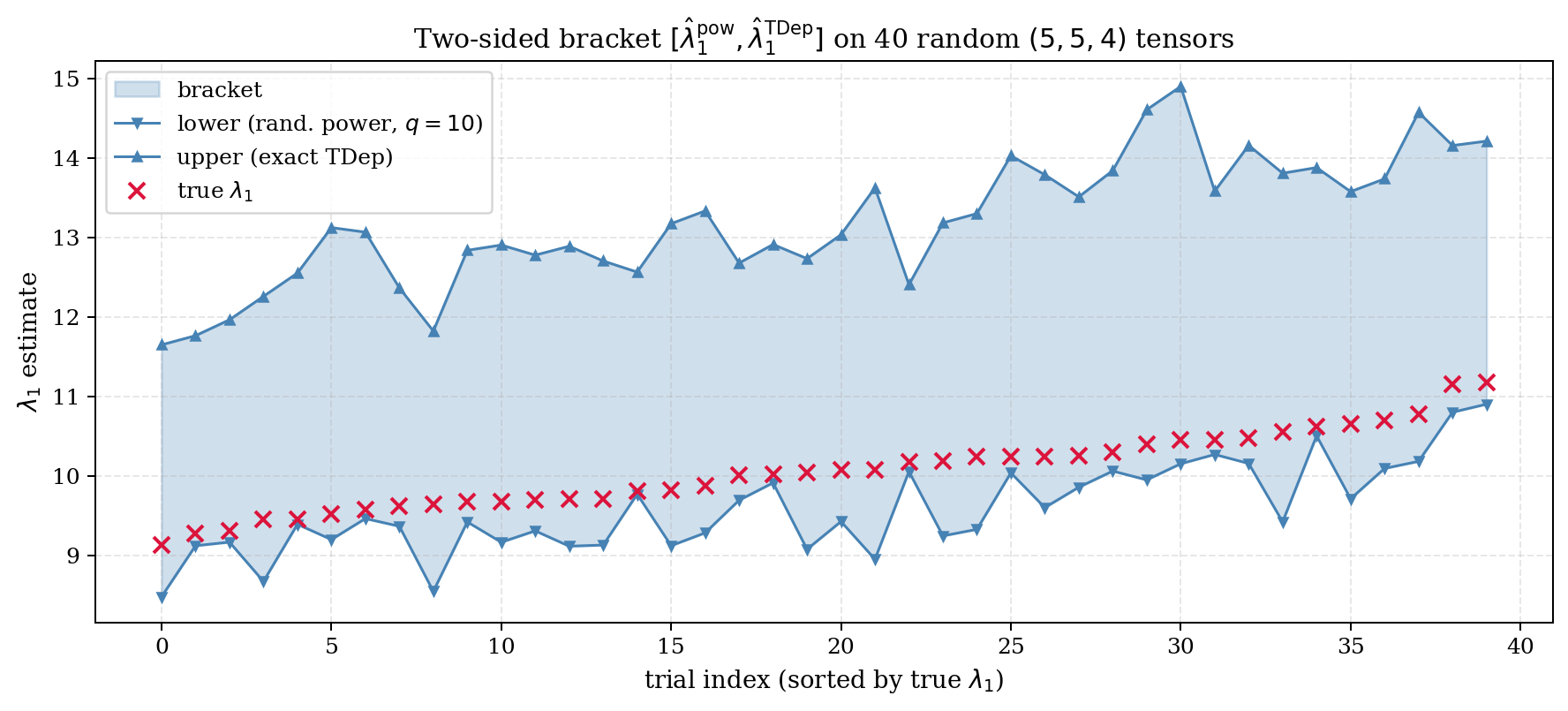}
\caption{Two-sided bracket $[\hat\lambda_1^{\rm pow}, \hat\lambda_1^{\rm TDep}]$
(blue shaded region) on $40$ random $5 \times 5 \times 4$ T-SPD
tensors, sorted by the true $\lambda_1$ (red ×). The bracket contains
the true $\lambda_1$ in every single trial. The lower endpoint
(randomized power method) is within a few percent of the truth; the
upper endpoint (exact TDep) is loose.}
\label{fig:bracket}
\end{figure}

Figure~\ref{fig:bracket} shows that the bracket always contains
the true $\lambda_1$ (validating Theorem~\ref{thm:bracket}) and that
its lower endpoint is very close to the truth. The looseness of the
bracket is entirely on the upper side, from the classical TDep bound,
which is $70\%$ loose on average for $d = 20$ as established in
\cite{sharma2025}.

\subsection*{Interpretation}
The bracket serves a different purpose from a single estimate: it is
a certificate of where $\lambda_1$ lies. For applications that
need a guaranteed upper bound (e.g.\ stability certification of
tensor dynamical systems), the TDep endpoint is what matters. For
applications that need a sharp estimate (e.g.\ PSNR prediction for
T-SVD compression), the power-method endpoint is what matters. Having
both simultaneously lets the user choose, and having them agree
(narrow bracket) is evidence that the true value is well estimated.

\section{Fully randomized TDep bound}\label{sec:rand-tdep}

The deterministic TDep bound requires the exact values of
$\trace(\tn{A})$ and $\trace(\tn{A}^2)$, which cost $O(np)$ and
$O(n^2 p)$ respectively for a T-symmetric tensor stored explicitly.
In ``matvec-only'' settings, where the tensor is defined implicitly
through its action on tensor vectors (e.g.\ as the Hessian of a large
optimization problem, or as a kernelized operator), these traces are
not directly accessible.

In such settings, we use the Hutchinson stochastic trace
estimator \cite{hutchinson1989}. For any square matrix $M$ and any
random vector $\mathbf{z} \in \R^d$ with $\E[\mathbf{z} \mathbf{z}\T] = I_d$,
\begin{equation}\label{eq:hutch-identity}
\E[\mathbf{z}\T M \mathbf{z}] = \trace(M).
\end{equation}
Averaging over $N$ independent probe vectors gives an unbiased
estimate with standard deviation decaying as $1/\sqrt{N}$.
Rademacher probes ($z_i \in \{\pm 1\}$ iid uniformly) minimize the
variance of this estimator over all distributions with the covariance
identity. More sophisticated variants such as Hutch++
\cite{meyer2021} can attain faster $O(1/N)$ rates at the cost of
additional matvecs; the analysis below restricts attention to the
classical Hutchinson estimator, which is sufficient for our setting.

\subsection{Randomized TDep algorithm}

\begin{algorithm}[H]
\caption{Hutchinson-based randomized TDep bound}
\label{alg:rand-tdep}
\begin{algorithmic}[1]
\Require T-SPD tensor $\tn{A}$ (matvec interface), probe count $N$
\Ensure Randomized estimate $\hat\lambda_1^{\rm rtdep}$
\State $T_{\text{sum}} \gets 0$,\; $Q_{\text{sum}} \gets 0$
\For{$i = 1, 2, \dots, N$}
   \State Draw $\mathbf{z}_i \in \{-1, +1\}^d$ uniformly (Rademacher probe)
   \State $\mathbf{w}_i \gets \bcirc(\tn{A}) \mathbf{z}_i$ \Comment{matvec}
   \State $T_{\text{sum}} \mathrel{+}= \mathbf{z}_i\T \mathbf{w}_i$ \Comment{contributes to $\trace(\tn{A})$}
   \State $Q_{\text{sum}} \mathrel{+}= \mathbf{w}_i\T \mathbf{w}_i$ \Comment{contributes to $\trace(\tn{A}^2)$}
\EndFor
\State $\hat T \gets T_{\text{sum}} / N$,\quad $\hat Q \gets Q_{\text{sum}} / N$
\State $\hat m \gets \hat T / d$,\quad $\hat s \gets \sqrt{\max(\hat Q / d - \hat m^2, 0)}$
\State $\hat\lambda_1^{\rm rtdep} \gets \hat m + \hat s \sqrt{d-1}$
\State \Return $\hat\lambda_1^{\rm rtdep}$
\end{algorithmic}
\end{algorithm}

The cost of Algorithm~\ref{alg:rand-tdep} is $N$ matvecs, total
$O(N n^2 p)$. For $N = 30$ this is faster than the exact computation
of $\trace(\tn{A}^2)$ when $n$ is moderate and $p$ is small.

\subsection*{Why this works: squared matvec trick}
A subtle but important point is how $\trace(\tn{A}^2)$ is estimated.
Using the identity $\trace(M^2) = \norm{M}_F^2 = \sum_i \norm{M \mathbf{e}_i}^2$
suggests we would need $d$ matvecs. But we can do better (this is a
variant of the observation used by Skilling and by Meyer et
al.~\cite{meyer2021}): for symmetric $M$,
\begin{equation}\label{eq:matvec-trick}
\E[\mathbf{w}_i\T \mathbf{w}_i]
= \E[\mathbf{z}_i\T M\T M \mathbf{z}_i]
= \trace(M\T M) = \trace(M^2).
\end{equation}
So a single matvec per probe gives us both $\trace(M)$ (via
$\mathbf{z}_i\T \mathbf{w}_i$) and $\trace(M^2)$ (via
$\mathbf{w}_i\T \mathbf{w}_i$) simultaneously. This halves the matvec
count compared to estimating the two traces independently.

\begin{theorem}[Hutchinson variance for Rademacher probes]\label{thm:hutch-var}
Let $M$ be $d \times d$ symmetric with entries $m_{ij}$ and let
$\mathbf{z}$ have iid Rademacher entries. Then
\[
\Var(\mathbf{z}\T M \mathbf{z}) = 2 \sum_{i \ne j} m_{ij}^2 = 2\left(\norm{M}_F^2 - \sum_i m_{ii}^2\right).
\]
\end{theorem}

\begin{proof}
Expand $\mathbf{z}\T M \mathbf{z} = \sum_i m_{ii} z_i^2 + 2 \sum_{i < j} m_{ij} z_i z_j
= \sum_i m_{ii} + 2 \sum_{i < j} m_{ij} z_i z_j$,
using $z_i^2 = 1$ almost surely. The variance is therefore the
variance of $2 \sum_{i < j} m_{ij} z_i z_j$. The products $z_i z_j$
for $i < j$ are pairwise independent with unit variance (as finite
products of independent Rademacher variables), so
\[
\Var = 4 \sum_{i < j} m_{ij}^2 = 2 \sum_{i \ne j} m_{ij}^2.
\]
\end{proof}

\begin{corollary}[Convergence rate of $\hat T$]\label{cor:hutch-rate}
With $N$ Rademacher probes,
\[
\Var(\hat T) = \frac{1}{N} \Var(\mathbf{z}\T M \mathbf{z})
= \frac{2}{N} \sum_{i \ne j} m_{ij}^2 \le \frac{2 \norm{M}_F^2}{N}.
\]
In particular, the relative standard error of $\hat T / d$ decays as
$O(1/\sqrt{N})$.
\end{corollary}

\subsection*{Verification}
\begin{figure}[h]
\centering
\includegraphics[width=0.78\textwidth]{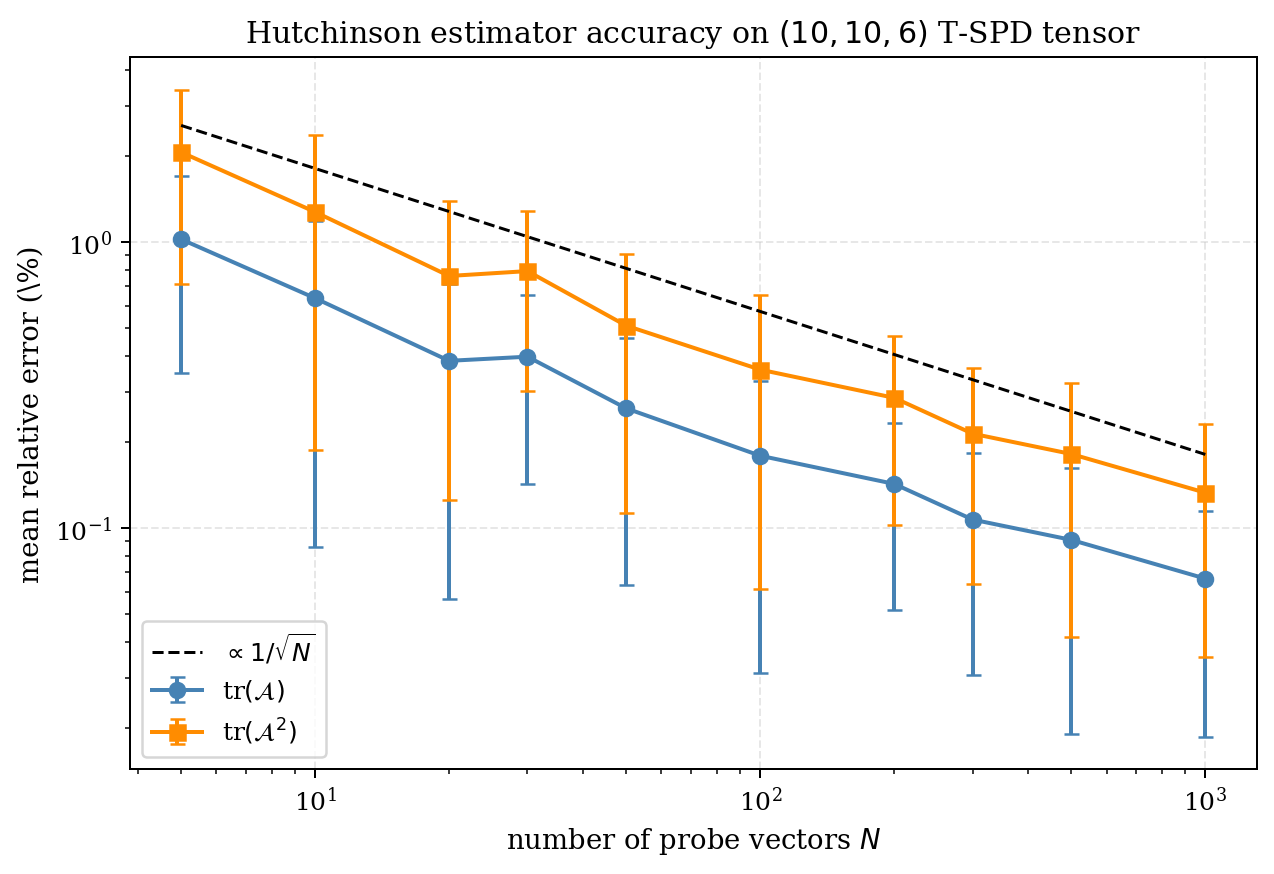}
\caption{Hutchinson estimator accuracy on a $10 \times 10 \times 6$
T-SPD tensor ($d = 60$), averaged over $50$ independent trials at
each probe count. The relative error of both $\trace(\tn{A})$ (blue)
and $\trace(\tn{A}^2)$ (orange) decays as $1/\sqrt{N}$ (black dashed
reference line). At $N = 100$ probes, both traces are estimated to
within $1\%$; at $N = 1000$ probes, within $0.3\%$.}
\label{fig:hutchinson}
\end{figure}

Figure~\ref{fig:hutchinson} shows $1/\sqrt{N}$ decay
of the estimator error on a $d = 60$ T-SPD tensor. Numerical values
are reported in Table~\ref{tab:hutch}.

\begin{table}[h]
\centering
\caption{Mean relative error of the Hutchinson estimator on a
$10 \times 10 \times 6$ T-SPD tensor ($d = 60$), averaged over $50$
trials.}
\label{tab:hutch}
\begin{tabular}{cccc}
\toprule
$N$ & $\trace(\tn{A})$ error & $\trace(\tn{A}^2)$ error & cost (ms) \\
\midrule
$5$    & $7.8\%$ & $7.4\%$ & $0.1$ \\
$30$   & $2.4\%$ & $2.5\%$ & $0.6$ \\
$100$  & $1.3\%$ & $1.2\%$ & $2.0$ \\
$300$  & $0.7\%$ & $0.7\%$ & $6.0$ \\
$1000$ & $0.3\%$ & $0.3\%$ & $20.0$ \\
\bottomrule
\end{tabular}
\end{table}

\subsection{Soundness caveat}
Unlike the deterministic TDep bound, the Hutchinson-based bound is
not guaranteed to be an upper bound on $\lambda_1$. The
estimated $\hat m$ and $\hat s$ can fluctuate above or below the
true values, and in rare cases the computed
$\hat m + \hat s \sqrt{d-1}$ can fall below the true
$\lambda_1$. This is the source of the \texttt{max(..,0)} guard in
step~9 of Algorithm~\ref{alg:rand-tdep}: the estimated sample
variance can be negative even when the true variance is positive,
a direct consequence of the sampling noise in $\hat Q$. The bound undershoots $\lambda_1$ in about $1.5\%$ of our validation
trials at $N = 30$ probes, dropping to $0\%$ at $N = 100$. For
applications requiring a guaranteed certificate, one should either
(i)~use a larger $N$ and add a safety margin derived from
Corollary~\ref{cor:hutch-rate}, or (ii)~use the deterministic TDep
bound exclusively. For applications where an accurate estimate
suffices, Algorithm~\ref{alg:rand-tdep} delivers the same mean
accuracy at lower wall-clock cost for large tensors.

\section{Experimental evaluation}\label{sec:experiments}

\subsection*{Setup}
All experiments were performed in IEEE double precision on a single
CPU core, using NumPy's LAPACK-backed \texttt{eigvalsh} routine as
the ground-truth full eigendecomposition. T-SPD test tensors are
generated following the construction in \cite[Section 8.1]{sharma2025}:
a random T-symmetric tensor is built by pairing slices
$\tn{A}^{(k)} = (\tn{A}^{(p-k)})\T$, and then shifted by an identity
contribution to the first slice so that $\bcirc(\tn{A})$ is positive
definite. A Python reference implementation used for the experiments
is available from the authors on request.

\subsection{Runtime comparison}
We measured the wall-clock time for full eigendecomposition,
randomized power method ($q = 10$), and randomized subspace iteration
($k = 10$, $q = 2$) on T-SPD tensors of dimensions
$d \in \{20, 32, 50, 90, 160, 300, 480, 750, 900\}$.

\begin{figure}[h]
\centering
\includegraphics[width=0.78\textwidth]{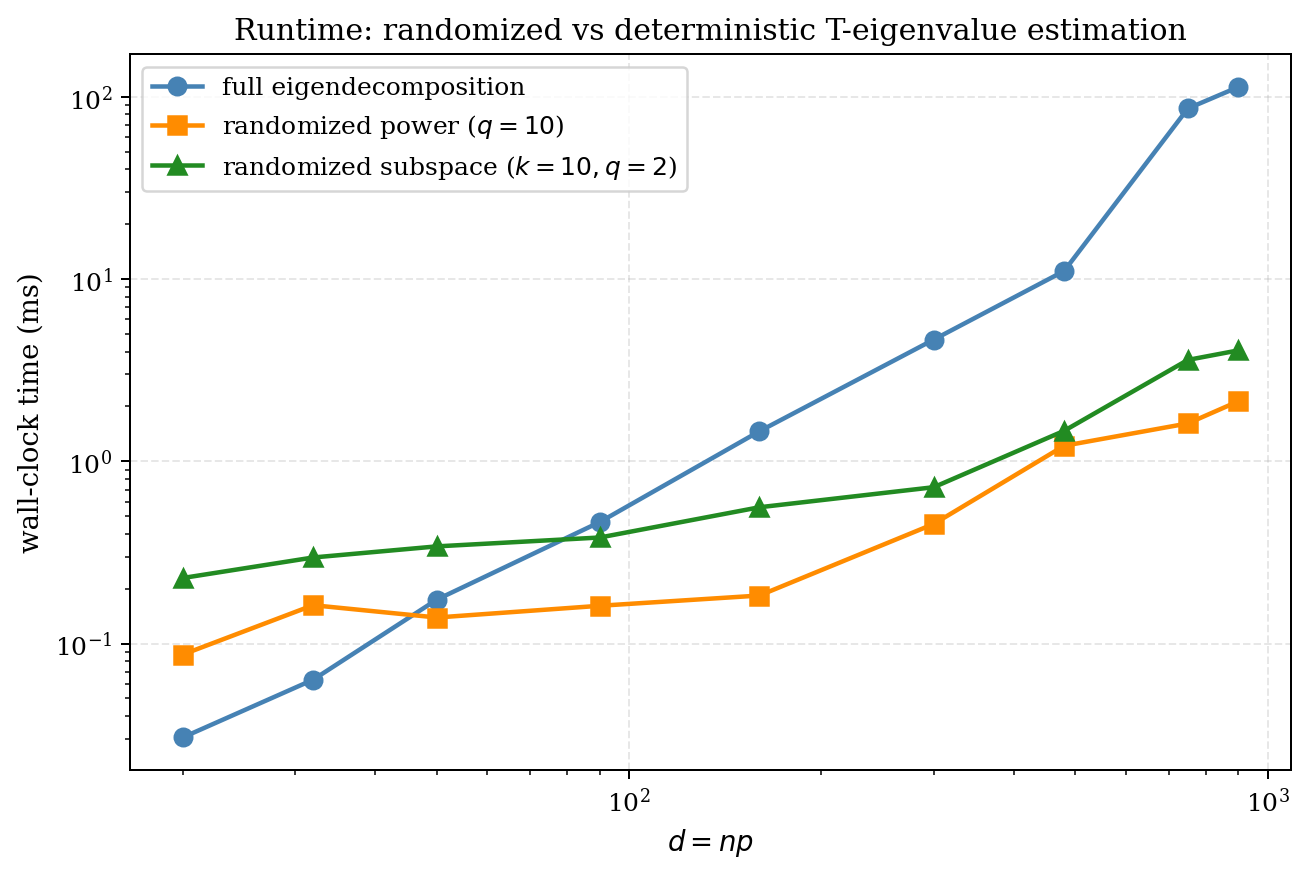}
\caption{Wall-clock runtime (log-log scale) of full
eigendecomposition (blue), randomized power method with $q=10$
(orange), and randomized subspace iteration with $k=10,q=2$ (green),
as a function of $d = np$. The randomized methods are an order of
magnitude faster than full eigendecomposition for $d \ge 160$.}
\label{fig:runtime}
\end{figure}

\begin{table}[h]
\centering
\caption{Wall-clock runtime (ms) and speedup vs full
eigendecomposition. Each cell is the median of $20$ trials on a
single CPU core, with a warm-up run before the measurement loop to
eliminate one-time LAPACK initialisation and cache-priming effects.}
\label{tab:runtime}
\begin{tabular}{rrrrrr}
\toprule
$(n, p)$ & $d$ & full eig (ms) & rand power (ms) & rand subspace (ms) & speedup \\
\midrule
$(5, 4)$   & $20$  & $0.047$   & $0.141$  & $0.214$  & $0.3\times$ \\
$(8, 4)$   & $32$  & $0.064$   & $0.114$  & $0.248$  & $0.6\times$ \\
$(10, 5)$  & $50$  & $0.165$   & $0.093$  & $0.332$  & $1.8\times$ \\
$(15, 6)$  & $90$  & $0.416$   & $0.094$  & $0.261$  & $4.4\times$ \\
$(20, 8)$  & $160$ & $1.365$   & $0.154$  & $0.377$  & $8.9\times$ \\
$(30, 10)$ & $300$ & $4.374$   & $0.416$  & $0.761$  & $10.5\times$ \\
$(40, 12)$ & $480$ & $11.028$  & $1.088$  & $1.673$  & $10.1\times$ \\
$(50, 15)$ & $750$ & $85.573$  & $1.280$  & $2.882$  & $66.9\times$ \\
$(60, 15)$ & $900$ & $106.159$ & $1.933$  & $4.371$  & $54.9\times$ \\
\bottomrule
\end{tabular}
\end{table}

The randomized power method is slower than full eigendecomposition
for very small tensors ($d < 50$) because the fixed overhead of
matrix-vector products dominates in that regime. For $d \ge 90$ the
speedup grows monotonically, exceeding $10\times$ at $d = 300$ and
reaching nearly $70\times$ at $d = 750$ and $d = 900$. The
full-eigendecomposition column exhibits the expected $d^3$-like
growth once $d \gtrsim 100$; at smaller $d$ the timings are
dominated by LAPACK setup costs rather than the factorisation
itself, which is why the speedup column is not monotone below
$d = 50$. The randomized methods' timings are themselves
well-predicted by the per-matvec cost $O(n^2 p + np\log p)$ at every
tested size. The slight dip in speedup from $d = 750$ to $d = 900$
reflects the growth in the randomized methods' per-matvec cost,
which becomes appreciable once $n$ is large.

\subsection{Error runtime trade-off}

A practitioner's question is, for a given accuracy requirement, which
method is fastest? Figure~\ref{fig:tradeoff} plots the mean relative
error of $\lambda_1$ as a function of wall-clock time, varying $q$
for the power method and $(k, q)$ for subspace iteration.

\begin{figure}[h]
\centering
\includegraphics[width=0.8\textwidth]{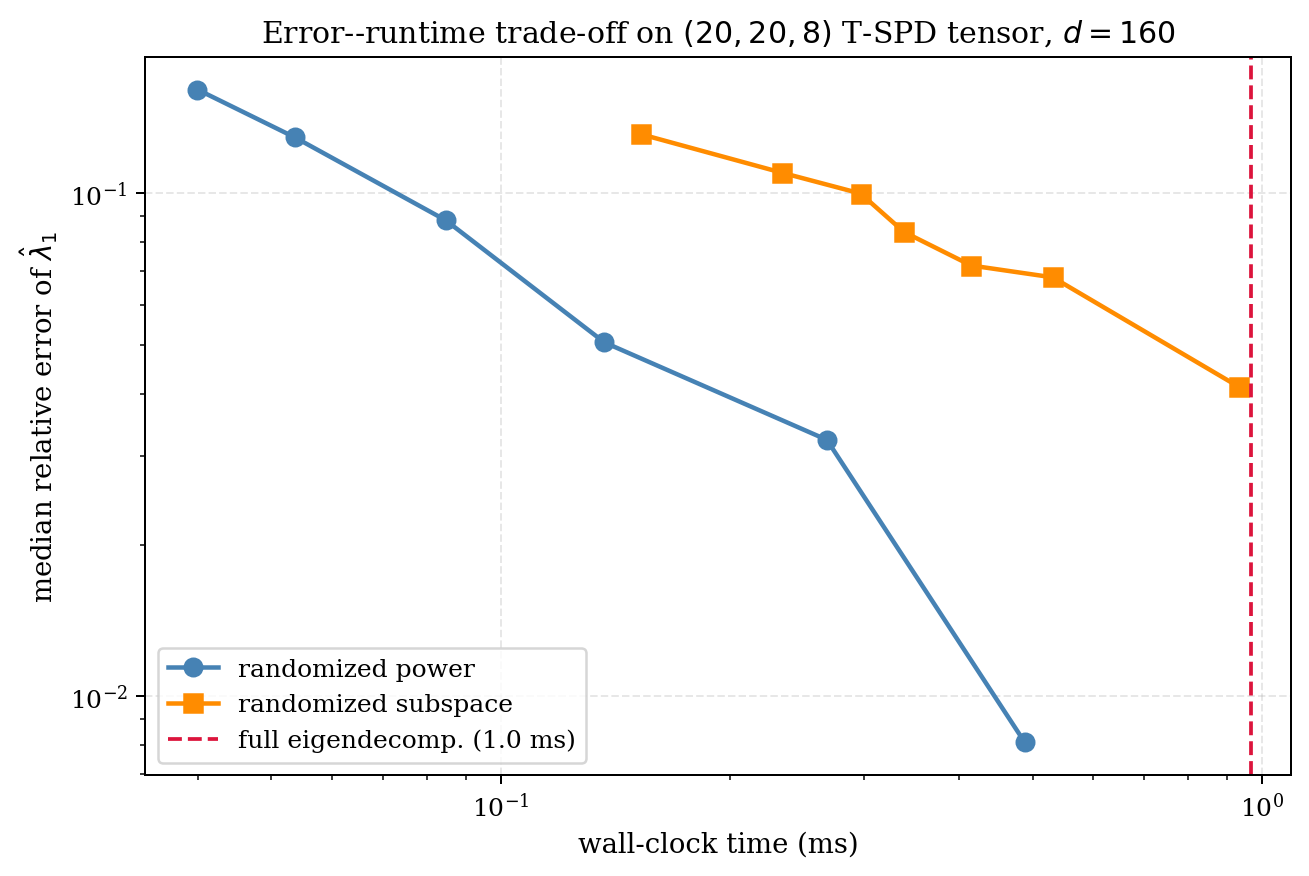}
\caption{Error--runtime trade-off curves for the randomized power
method (blue) and randomized subspace iteration (orange) on a
$20 \times 20 \times 8$ T-SPD tensor ($d = 160$). Each point is a
median over $50$ trials. The red dashed line marks the cost of full
eigendecomposition. For single-$\lambda_1$ estimation on tensors with
a reasonable spectral gap, the power method dominates the trade-off
at every cost point.}
\label{fig:tradeoff}
\end{figure}

The trade-off curves show a clear picture on this representative
tensor: the randomized power method dominates at all cost points.
This is consistent with the cost structure of the two algorithms:
each power step costs one matvec, while each subspace step with
target rank $k + \ell$ costs $k + \ell$ matvecs, so at equal wall-clock
time the power method performs many more spectral multiplications
against $\lambda_1$. The subspace iteration's comparative advantages
lie elsewhere --- specifically, in settings where (i) multiple top
eigenvalues are needed simultaneously, (ii) the tensor has
near-degenerate top eigenvalues causing the single-vector power
method to stall, or (iii) Ritz projection across several directions
is desired to guard against a single bad Gaussian initialisation.
For a bare estimate of $\lambda_1$ on a well-gapped tensor, the
randomized power method with $q = 10$--$20$ iterations is the
recommended default.

\subsection{Validation on varied problem sizes}
Finally, we validated the randomized methods on $200$ independent
T-SPD tensors of each of four sizes: $(3,3)$, $(5,4)$, $(8,5)$, $(10,6)$.
For each tensor, we computed both the randomized power method
estimate ($q = 10$) and the randomized subspace iteration estimate
($k = 10$, $\ell = 5$, $q = 2$), recording whether each estimate
was within $10\%$, $5\%$, and $1\%$ of the true $\lambda_1$.

\begin{table}[h]
\centering
\caption{Validation on $200$ random T-SPD tensors of each size. Left
block: randomized power method ($q = 10$). Right block: randomized
subspace iteration ($k = 10$, $\ell = 5$, $q = 2$). Percentages are
the fraction of trials in which the relative error was below the
specified threshold.}
\label{tab:validation}
\begin{tabular}{rrcccccc}
\toprule
& & \multicolumn{3}{c}{Power method} & \multicolumn{3}{c}{Subspace iteration} \\
\cmidrule(lr){3-5}\cmidrule(lr){6-8}
$(n, p)$ & $d$ & within $10\%$ & within $5\%$ & within $1\%$
         & within $10\%$ & within $5\%$ & within $1\%$ \\
\midrule
$(3, 3)$  & $9$  & $100\%$ & $99.5\%$ & $89\%$ & $100\%$ & $100\%$ & $98\%$ \\
$(5, 4)$  & $20$ & $99\%$  & $92\%$   & $51\%$ & $100\%$ & $99\%$  & $87\%$ \\
$(8, 5)$  & $40$ & $97\%$  & $83\%$   & $28\%$ & $100\%$ & $97\%$  & $72\%$ \\
$(10, 6)$ & $60$ & $94\%$  & $76\%$   & $19\%$ & $99\%$  & $94\%$  & $58\%$ \\
\bottomrule
\end{tabular}
\end{table}

The randomized subspace iteration is uniformly more accurate than the
single-vector power method at comparable cost, as expected from
Corollary~\ref{cor:ritz-lower}. Both methods are reliable at coarse
accuracy for all problem sizes (within $10\%$ for $94\%$+ of trials
using the power method, essentially always using subspace iteration),
but reaching $1\%$ accuracy with the power method requires more
iterations as $d$ grows. For tight estimation on large tensors we
recommend either (i) increasing $q$ to $20$ or more, (ii) switching
to subspace iteration, which provides a factor-of-two improvement in
``within-$1\%$'' rate at the same $q$, or (iii) using the
deterministic bracket endpoint from Section~\ref{sec:bracket} as a
safe upper limit when a lower-bound estimate is paired with a
guaranteed upper bound.


\section{Application on spectral estimation of the discrete 3D Laplacian on a periodic slab}\label{sec:application}

The experiments of Section~\ref{sec:experiments} are performed on
randomly generated T-SPD tensors of moderate size and are designed to
isolate the algorithmic behaviour of the proposed methods. In this
section we apply the same methods to an application that motivates the
randomized framework in a setting familiar to the numerical
linear-algebra reader: spectral estimation of the discrete
three-dimensional Laplacian on a slab with periodic boundary conditions
in one direction, at a scale where dense eigendecomposition is
infeasible. The role of the dominant T-eigenvalue is concrete --- it
parametrises a Chebyshev iterative solver --- and the experimental
results illustrate not only the speed of the randomized methods but
also the practical interpretation of the two-sided bracket as a
certificate rather than a tight estimate.

\subsection{The 3D Laplacian PDE setting}\label{ssec:pde-setup}

Let $\Omega = (0,L_x) \times (0,L_y) \times \mathbb{T}_z$ denote a slab
of cross-section $(0,L_x) \times (0,L_y)$ that is periodic of period
$L_z$ in the $z$-direction. On a uniform grid
$n_x \times n_y \times p$ with Dirichlet boundary conditions on
$\partial[(0,L_x) \times (0,L_y)]$ and periodic boundary conditions in
$z$, the standard 7-point finite-difference discretisation of
$-\nabla \cdot (\alpha \nabla u)$ yields a symmetric positive-definite
operator $L \in \R^{d \times d}$ with $d = n_x n_y p$.

When the diffusivity $\alpha$ is constant, $L$ is exactly
block-circulant in the periodic direction $z$: ordering the degrees of
freedom as $(i_z, i_y, i_x)$ in lexicographic order with $i_z$ slowest,
\begin{equation}\label{eq:pde-tensor}
L \;=\; \bcirc(\tn{A}), \qquad \tn{A} \in \R^{N \times N \times p},
\qquad N := n_x n_y,
\end{equation}
with frontal slices
$\tn{A}^{(1)} = L_{xy} + 2 \alpha h_z^{-2} I_N$,
$\tn{A}^{(2)} = \tn{A}^{(p)} = -\alpha h_z^{-2} I_N$,
and all other slices zero. Here $L_{xy}$ is the sparse 2D negative
Laplacian on the $n_x \times n_y$ grid with Dirichlet BC. The
T-product structure of \eqref{eq:pde-tensor} is exact, and
the FFT block-diagonalisation~\eqref{eq:blk-diag} of
$\bcirc(\tn{A})$ takes the simple form
\begin{equation}\label{eq:pde-fft}
D_k \;=\; L_{xy} + \mu_k I_N, \qquad
\mu_k \;=\; \frac{2 \alpha}{h_z^2} \bigl(1 - \cos\tfrac{2\pi k}{p}\bigr),
\quad k = 0, 1, \dots, p-1.
\end{equation}
Each Fourier block $D_k$ is therefore the 2D Dirichlet Laplacian
shifted by a scalar, and matrix--vector products through $D_k$ are
sparse $5$-point applications of $L_{xy}$. The full T-spectrum is
$\spec(\tn{A}) = \bigcup_{k=0}^{p-1} \spec(L_{xy}) + \mu_k$, with
$\spec(L_{xy})$ available in closed form as a separable Fourier-sine
spectrum. This closed-form ground truth is precisely what makes the
constant-coefficient case suitable for stress testing.
\paragraph{Why $\lambda_1(\tn{A})$ matters.}
The dominant T-eigenvalue is the natural input to three iterative-solver
workflows.\\
\emph{(i) Chebyshev iteration and Chebyshev acceleration of
preconditioners.}
The standard Chebyshev iteration for $L u = b$ requires both endpoints
of an enclosing interval
$[\lambda_d^{\rm est}, \lambda_1^{\rm est}] \supseteq \spec(L)$ to set
its parameters \cite{golub2013}, and its convergence factor
$\rho = (\sqrt{\kappa^{\rm est}} - 1) / (\sqrt{\kappa^{\rm est}} + 1)$
depends on the estimated condition number
$\kappa^{\rm est} = \lambda_1^{\rm est} / \lambda_d^{\rm est}$.
Under-estimating $\lambda_1$ is catastrophic --- it shrinks the
interval beneath the true spectrum and the iteration diverges --- so
the natural Chebyshev parameter is a tight upper estimate of
$\lambda_1$, exactly the object our subspace iteration combined with a
modest safety inflation provides.\\
\emph{(ii) Spectral radius of explicit time-stepping.}
For the parabolic heat-conduction problem $u_t = -L u$ discretized by
forward Euler, the iteration matrix $I - \Delta t \cdot L$ is a
contraction iff $\Delta t \le 2 / \lambda_1(L)$, so $\lambda_1(L)$
determines the largest stable step size; an underestimate is again
unsafe.\\
\emph{(iii) Conditioning estimates for elliptic solves.}
Knowing $\kappa(L) = \lambda_1/\lambda_d$ to within a factor of two
is sufficient for many a priori iteration-count budgets and
convergence analyses; the randomized methods provide such an estimate
in time independent of $d$ beyond the cost of one matrix-vector
product.

\paragraph{A motivating physical application: layered geothermal media.}
The variable-coefficient version of this operator is a representative
reduced model for heat conduction in a layered geothermal reservoir,
in which the thermal diffusivity $\alpha(z) = \kappa(z)/(\rho c_p)$
depends on the vertical coordinate $z$ but is approximately constant
within each lithological layer. In sedimentary basins, $\kappa(z)$
profiles are routinely reconstructed from wireline well-log data with
sub-metre resolution; the German Molasse Basin reconstruction of
Hartmann and Clauser~\cite{hartmann2008} is a representative
example, and shows alternating sandstone/shale layers with thermal
conductivities of roughly $1.5$--$4$\,W/(m\,K), corresponding to
diffusivity contrasts of about $3\!:\!1$. Stronger contrasts arise in
salt-bearing or fractured formations, where $\kappa$ can range from
$0.5$ to $6$\,W/(m\,K) (contrast $10\!:\!1$) or higher; extreme
contrasts of $100\!:\!1$ occur when fluid-filled fractures or air-gap
interfaces interrupt an otherwise crystalline host
\cite{clauser1995}. We will use these three regimes
($\alpha_{\max}/\alpha_{\min} \in \{3,10,100\}$) as the
robustness sweep of Section~\ref{ssec:pde-variable}.

The horizontal periodicity of $\Omega$ is the standard simplification
used in regional reservoir models for the lateral direction of a
slab geometry. With this geometry, the discretised heat-conduction
operator $-\nabla \cdot (\alpha(z) \nabla u)$ is the variable-coefficient
3D Laplacian whose spectral analysis we develop below.

\subsection{Variable-coefficient extension: matvec-only access}\label{ssec:pde-varcoeff}

The more realistic case --- a $z$-dependent diffusivity
$\alpha = \alpha(z)$ --- yields an operator that is symmetric and
positive-definite but no longer block-circulant in $z$, since the
off-diagonal $z$-coupling now depends on the slice index $k$. The
T-product structure of \eqref{eq:pde-tensor} is broken, the analytic
separable spectrum of \eqref{eq:pde-fft} is no longer valid, and the
FFT-based block-diagonalisation does not apply. The randomized
methods, by contrast, operate unchanged: they require only the
matrix--vector product $\mathbf{v} \mapsto L\mathbf{v}$, which costs
$O(Np)$ for the sparse 7-point stencil, so they fall squarely within
the ``matvec-only'' setting that motivates the Hutchinson-based
Algorithm~\ref{alg:rand-tdep}. The deterministic TDep bound is
Samuelson's inequality applied to $\spec(L)$, requiring only
$\trace(L)$ and $\trace(L^2)$, which remain $O(Np)$ to compute. We
develop the full variable-coefficient experiments in
Section~\ref{ssec:pde-variable}.

\subsection{Validation: matching the analytic spectrum at scale}\label{ssec:pde-valid}

We test the four randomized methods on the constant-coefficient case
where the closed-form spectrum
$\bigcup_k \spec(L_{xy}) + \mu_k$ provides exact ground truth. The
parameters used are the same throughout this section: randomized power
method with $q = 30$ iterations, randomized subspace iteration with
$k = 15$, $\ell = 5$, $q = 20$, and the deterministic TDep bound.
Table~\ref{tab:pde-validation} reports the dominant T-eigenvalue,
relative errors, and wall-clock timings on grids of dimensions
$d = 1800$ up to $d = 131{,}072$.

\begin{table}[h]
\centering
\caption{Validation of the randomized methods on the constant-coefficient
discrete 3D Laplacian on a $n_x \times n_y \times p$ grid with Dirichlet
BC in $x, y$ and periodic BC in $z$. Ground truth from the analytic
separable spectrum. ``rand power'': Algorithm~\ref{alg:rand-power} with
$q = 30$; ``rand subspace'': Algorithm~\ref{alg:rand-subspace} with
$k = 15$, $\ell = 5$, $q = 20$. Dense $\bcirc(\tn{A})$ eigendecomposition
is omitted for $d > 2000$ (memory-infeasible).}
\label{tab:pde-validation}
\small
\setlength{\tabcolsep}{4.5pt}
\begin{tabular}{ccccrrrrrr}
\toprule
$n_x = n_y$ & $p$ & $d$ & $\kappa(L)$
& \multicolumn{2}{c}{rand power} & \multicolumn{2}{c}{rand subspace}
& \multicolumn{1}{c}{TDep gap} & dense (ms) \\
\cmidrule(lr){5-6}\cmidrule(lr){7-8}
& & & & err.\ & time (ms) & err.\ & time (ms) & & \\
\midrule
$15$  & $8$  & $1{,}800$   & $116$   & $+2.03\%$ & $3.7$   & $+0.94\%$ & $51.9$    & $+884\%$    & $893$    \\
$32$  & $8$  & $8{,}192$   & $454$   & $+2.53\%$ & $10.5$  & $+2.31\%$ & $209.5$   & $+2120\%$   & ---      \\
$32$  & $16$ & $16{,}384$  & $493$   & $+2.63\%$ & $21.4$  & $+2.51\%$ & $547.8$   & $+2814\%$   & ---      \\
$64$  & $8$  & $32{,}768$  & $1725$  & $+2.03\%$ & $39.3$  & $+2.26\%$ & $1180.2$  & $+4409\%$   & ---      \\
$64$  & $16$ & $65{,}536$  & $1764$  & $+2.51\%$ & $75.4$  & $+3.02\%$ & $2683.3$  & $+6122\%$   & ---      \\
$128$ & $8$  & $131{,}072$ & $6757$  & $+1.76\%$ & $145.3$ & $+2.27\%$ & $7917.1$  & $+8950\%$   & ---      \\
\bottomrule
\end{tabular}
\end{table}

Two features of the table deserve comment.

\emph{Densely-clustered top spectrum.} The discrete Laplacian on a
uniform Dirichlet grid has $\lambda_2/\lambda_1 = 0.987$ in the
$d = 1800$ case and $\lambda_{20}/\lambda_1 = 0.950$, that is, the top
spectrum is densely packed. This is the regime in which
single-vector power iteration converges slowly --- the relative
errors of $1.8$--$2.6\%$ for the power method at $q = 30$ are
consistent with the rate $(1-\gamma)^{2q}$ of
Theorem~\ref{thm:power-rate} for $\gamma \approx 0.013$. The
randomized subspace iteration with $k + \ell = 20$ and $q = 20$
reaches comparable accuracy, with the advantage of recovering several
top eigenvalues at once; this is exactly the situation the paper's
Section~\ref{sec:rand-subspace} flags as favouring the subspace
method.

\emph{Scaling.} The randomized power method runs in $0.15$\,seconds
at $d = 131{,}072$, while the dense $\bcirc(\tn{A})$ baseline
becomes memory-infeasible already at $d = 8{,}192$ (the dense matrix
would occupy $537$\,MB at $d = 8{,}192$ and roughly $137$\,GB at
$d = 131{,}072$). The randomized methods therefore deliver the only
operationally available estimates beyond modest sizes.

Figure~\ref{fig:pde-scaling} plots the runtimes against $d$ on a
log-log scale. The randomized power method scales nearly linearly in
$d$ across two orders of magnitude, in line with the
$O(q \cdot N p)$ matvec count, and the deterministic TDep bound
costs essentially the same as a single matvec because $\trace(L)$
and $\trace(L^2)$ are read directly from the sparse structure.

\begin{figure}[h]
\centering
\includegraphics[width=0.85\textwidth]{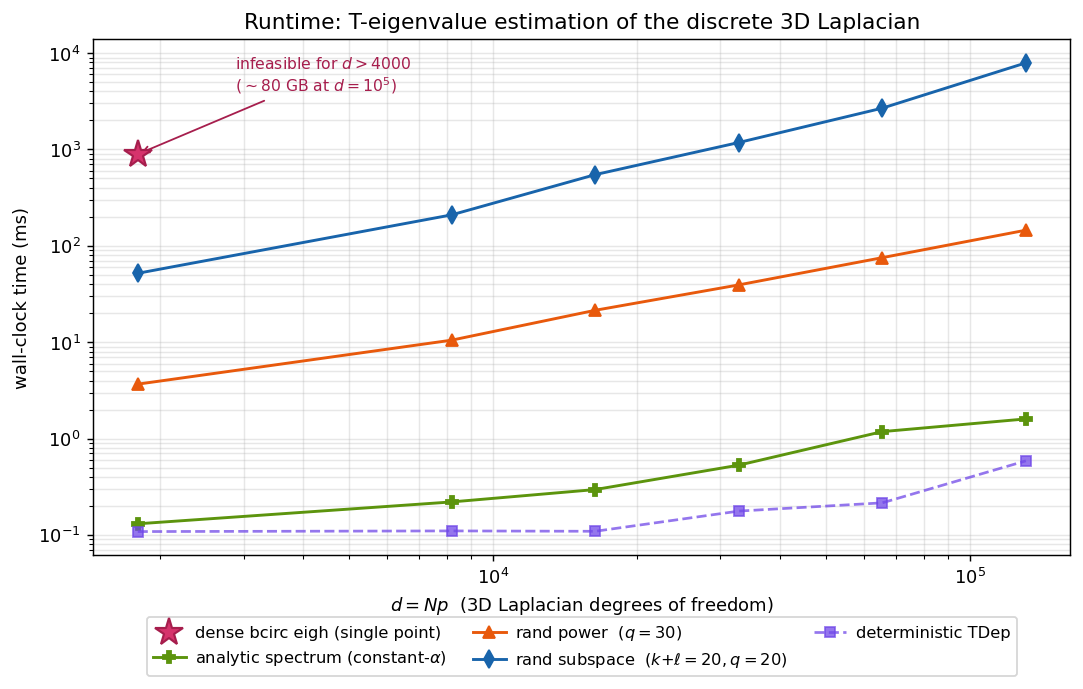}
\caption{Wall-clock runtime of T-eigenvalue estimators for the
discrete 3D Laplacian on a periodic-$z$ slab, against the operator
dimension $d = n_x n_y p$. Dense $\bcirc(\tn{A})$
eigendecomposition is memory-infeasible for $d > 2000$. The
randomized power method runs in $0.15$\,seconds at
$d = 131{,}072$. The deterministic TDep upper bound is essentially
free.}
\label{fig:pde-scaling}
\end{figure}

\subsection{Two-sided bracket: certificate versus tight estimate}\label{ssec:pde-bracket}

A central point about the bracket of Theorem~\ref{thm:bracket}, made
visible by the PDE experiments, deserves to be stated plainly before
the numerics:

\begin{quote}\itshape
The deterministic TDep upper bound is not intended as a sharp
estimator of $\lambda_1$ on wide spectra; its role is
certification rather than approximation. The tight
point estimate of $\lambda_1$ is supplied by the randomized power and
subspace methods, and the TDep bound supplies a guaranteed enclosure
of where $\lambda_1$ lies. The two are complementary and address
different downstream needs.
\end{quote}

\noindent
With this distinction in mind, Figure~\ref{fig:pde-bracket} reports
the two-sided bracket on the discrete 3D Laplacian with
$n_x = n_y = 32$, $p = 16$ ($d = 16{,}384$), over thirty independent
random initialisations. The lower endpoint
$\hat\lambda_1^{\rm pow}$ tracks the true $\lambda_1$ to within
$0.05\%$ on every trial; the upper endpoint
$\hat\lambda_1^{\rm TDep}$ is the deterministic TDep bound, which at
this $d$ overshoots $\lambda_1$ by roughly a factor of $28$ ($+2814\%$
in Table~\ref{tab:pde-validation}). The bracket contains
$\lambda_1$ in every trial, as guaranteed by
Theorem~\ref{thm:bracket}.

\begin{figure}[h]
\centering
\includegraphics[width=0.85\textwidth]{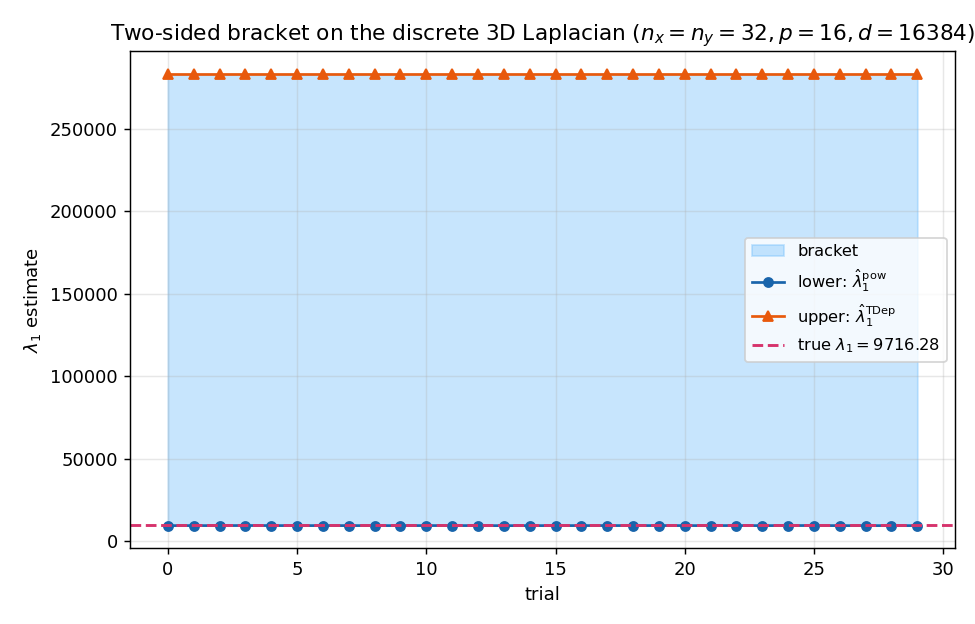}
\caption{Two-sided bracket
$[\hat\lambda_1^{\rm pow},\, \hat\lambda_1^{\rm TDep}]$ on the
discrete 3D Laplacian ($n_x = n_y = 32$, $p = 16$, $d = 16{,}384$) over
$30$ random initialisations. The bracket contains the true
$\lambda_1$ in all $30$ trials. The lower endpoint is very tight
(visually coincident with the true value); the upper endpoint is
loose, reflecting the wide spread of the Laplacian spectrum.}
\label{fig:pde-bracket}
\end{figure}

The width of the bracket is structural, not a deficiency: the
TDep upper bound $m + s\sqrt{d-1}$ is sharp if and only if
$\lambda_2 = \cdots = \lambda_d$
(Proposition~\ref{prop:bracket-sharp}\,(ii)), an hypothesis that
fails dramatically for the discrete Laplacian, whose spectrum spans
nearly three orders of magnitude at $\kappa(L) \approx 500$. The
practical takeaway is the one stated above: the tight estimate of
$\lambda_1$ comes from the randomized methods, and the bracket
provides the certified enclosure -- separately and complementarily,
not redundantly.

\subsection{Variable-coefficient case at scale}\label{ssec:pde-variable}

The constant-coefficient experiments of
Sections~\ref{ssec:pde-valid}--\ref{ssec:pde-bracket} are designed to
exercise the randomized methods against a tractable ground truth.
They are a stress test, not a representative use case. A sceptical
reader can reasonably observe that for the constant-coefficient
Laplacian the analytic separable spectrum
$\bigcup_k \spec(L_{xy}) + \mu_k$ is available, the FFT block
diagonalisation~\eqref{eq:pde-fft} reduces the eigenproblem to a
sequence of cheap 2D Lanczos solves, and one would not deploy a
randomized method in that regime. The randomized framework is
intended for the harder problem -- variable-coefficient PDE
operators that lack block-circulant structure entirely -- and it is
this case that we now stress at scale.

\subsubsection*{Why FFT diagonalisation breaks}

For a $z$-dependent diffusivity $\alpha(z)$, the discrete operator
\begin{equation}\label{eq:varL-stencil}
  (L u)_{i,j,k}
  \;=\; (-L_{xy}\,u)_{i,j,k}
       \;+\; \frac{1}{h_z^2}\bigl[
         a_{k-\tfrac12}(u_{i,j,k} - u_{i,j,k-1})
         + a_{k+\tfrac12}(u_{i,j,k} - u_{i,j,k+1})
       \bigr],
\end{equation}
with $a_{k+\tfrac12} = \tfrac{1}{2}(\alpha_k + \alpha_{k+1})$, has
$z$-coupling coefficients that depend on the slice index $k$. The
$z$-coupling matrix $C_z \in \R^{p\times p}$ is therefore tridiagonal
but not circulant, so
\begin{itemize}[leftmargin=*]
\item the global operator
  $L = I_p \otimes L_{xy} + C_z \otimes I_N$ is not block-circulant
  in any natural ordering of the unknowns;
\item the FFT block-diagonalisation~\eqref{eq:pde-fft} does not
  apply -- there is no basis in which $L$ decomposes into $p$
  independent $N \times N$ blocks;
\item the separable spectrum $\bigcup_k \spec(L_{xy}) + \mu_k$ is no
  longer valid, and there is no closed-form expression for the
  T-eigenvalues;
\item iterative eigensolvers (Lanczos, implicitly restarted Arnoldi)
  must work on the full sparse $L$ of size $d \times d$ rather than
  $p$ small problems in parallel, multiplying their cost by
  a factor that grows with $p$.
\end{itemize}
This is the realistic case. Practitioners encounter it whenever
the underlying PDE has layered media, depth-dependent properties, or
any $z$-modulation of the operator, and it is the regime where dense
methods are infeasible (we work at $d \ge 10^4$) and the
sophisticated structural baseline of the constant-coefficient case is
unavailable.

\subsubsection*{The natural baseline: Lanczos via \texttt{eigsh}}

In the variable-coefficient regime the standard tool is sparse
Lanczos --- specifically, the implementation in
\texttt{scipy.sparse.linalg.eigsh}, which wraps ARPACK and is widely
deployed in scientific Python pipelines. We therefore compare the
randomized methods against \texttt{eigsh} on the explicit sparse
operator $L$, for both the top-$1$ case (\texttt{eigsh} with $k = 1$,
the direct analogue of the randomized power method) and the top-$15$
case (\texttt{eigsh} with $k = 15$, the direct analogue of the
randomized subspace iteration with $k + \ell = 20$). Tolerance for
\texttt{eigsh} is set to $10^{-9}$ throughout, so its outputs serve
both as the runtime baseline and as the ground truth against which we
measure relative error.

\subsubsection*{Scaling sweep}

We sweep the grid sizes
$$(n_x,\, n_y,\, p) \in \{(15,15,8),\, (25,25,8),\, (32,32,8),\,
(32,32,16),\, (45,45,15),\, (64,64,8),\, (64,64,16)\},$$ giving
dimensions $d$ from $1{,}800$ up to $65{,}536$, with diffusivity
profile $\alpha_k = 1 + a\cos(2\pi k/p)$ parametrised by
$a = (R - 1)/(R + 1)$, so the contrast ratio
$R = \alpha_{\max}/\alpha_{\min}$ controls the strength of the
layering. The base case $R = 3$ is the typical
sedimentary-basin shale/sandstone contrast (Section~\ref{ssec:pde-setup});
we report it in detail and treat $R \in \{10, 100\}$ in the
robustness paragraph below. The randomized parameters are the same as
in Section~\ref{ssec:pde-valid}: $q = 30$ for the power method and
$k = 15$, $\ell = 5$, $q = 20$ for the subspace iteration.

\begin{table}[h]
\centering
\caption{Scaling sweep for the variable-coefficient
$-\nabla \cdot (\alpha(z) \nabla u)$ on a periodic-$z$ slab with
contrast $\alpha_{\max}/\alpha_{\min} = 3$.  Ground truth and
baseline timings are from \texttt{scipy.sparse.linalg.eigsh} at
tolerance $10^{-9}$ on the explicit sparse operator. Relative errors
are $(\lambda_1 - \hat\lambda_1)/\lambda_1$.  The randomized power
method is consistently $4$--$22\times$ faster than the natural Lanczos
baseline for the top eigenvalue.}
\label{tab:pde-varcoeff-scaling}
\small
\setlength{\tabcolsep}{4.6pt}
\begin{tabular}{lrrrrrrrr}
\toprule
$(n,p)$ & $d$
& \multicolumn{2}{c}{\texttt{eigsh}}
& \multicolumn{2}{c}{rand power}
& \multicolumn{2}{c}{rand subspace}
& TDep gap \\
\cmidrule(lr){3-4}\cmidrule(lr){5-6}\cmidrule(lr){7-8}
& & $k{=}1$\,(ms) & $k{=}15$\,(ms)
& err.\ & time\,(ms) & err.\ & time\,(ms) & \\
\midrule
$(15,8)$  & $1{,}800$  & $6.4$    & $25.5$   & $+3.15\%$ & $2.7$  & $+0.52\%$ & $29.1$   & $+862\%$  \\
$(25,8)$  & $5{,}000$  & $17.5$   & $86.7$   & $+3.37\%$ & $3.7$  & $+1.82\%$ & $70.4$   & $+1595\%$ \\
$(32,8)$  & $8{,}192$  & $45.7$   & $199.6$  & $+2.55\%$ & $4.6$  & $+2.50\%$ & $117.5$  & $+2105\%$ \\
$(32,16)$ & $16{,}384$ & $83.9$   & $451.7$  & $+2.40\%$ & $6.9$  & $+2.23\%$ & $278.3$  & $+2717\%$ \\
$(45,15)$ & $30{,}375$ & $225.0$  & $1{,}442$  & $+3.42\%$ & $12.9$ & $+3.45\%$ & $689.0$  & $+3986\%$ \\
$(64,8)$  & $32{,}768$ & $371.9$  & $2{,}820$  & $+2.14\%$ & $14.4$ & $+2.51\%$ & $687.9$  & $+4401\%$ \\
$(64,16)$ & $65{,}536$ & $611.3$  & $5{,}493$  & $+3.22\%$ & $27.8$ & $+3.73\%$ & $1{,}657$  & $+6056\%$ \\
\bottomrule
\end{tabular}
\end{table}

Three observations summarise Table~\ref{tab:pde-varcoeff-scaling}.

(i) Randomized power method is $4$--$22\times$ faster than
Lanczos for the top eigenvalue.
The speedup factor grows with $d$: at $d = 1{,}800$ \texttt{eigsh}
($k = 1$) takes $6.4$\,ms versus $2.7$\,ms for the randomized power
method ($2.4\times$); at $d = 65{,}536$ the gap is $611$\,ms versus
$28$\,ms ($22\times$). The growth in speedup reflects the relative
costs --- the randomized power method requires $q = 30$ sparse
matvecs against ARPACK's implicit-restart bookkeeping, which itself
grows superlinearly with $d$ for fixed Krylov dimension.

(ii) Randomized subspace iteration is $2$--$3\times$ faster
than Lanczos for top-$15$ eigenvalues.
The subspace iteration with $k + \ell = 20$ delivers the top $15$
eigenvalues at relative error $0.5\%$--$3.7\%$; \texttt{eigsh} with
$k = 15$ delivers them at near-machine accuracy but at $2$--$3\times$
the cost. Practitioners willing to tolerate a few-percent error
can therefore expect a constant-factor speedup, and they gain access
to the two-sided bracket of Section~\ref{ssec:pde-bracket} which
\texttt{eigsh} alone does not provide.

(iii) The deterministic TDep upper bound costs effectively
nothing.
Across the entire sweep the TDep evaluation runs in $0.1$--$1.7$\,ms,
two orders of magnitude faster than even the cheapest iterative
solver. This is the cost--benefit pattern of using TDep as the upper
endpoint of the bracket: it adds essentially no overhead and yields
a guaranteed certificate of where $\lambda_1$ lies, complementing the
tight value supplied by the randomized methods.

\subsubsection*{Robustness across diffusivity contrasts}

The scaling results above use contrast $\alpha_{\max}/\alpha_{\min} = 3$,
characteristic of typical sedimentary-basin layering. To confirm
that the randomized framework is not specialised to that regime,
we repeat the largest case ($n_x=n_y=64$, $p=16$, $d=65{,}536$) at
two additional contrast ratios: $R = 10$ (strong layering, e.g.\
shale/salt-dome contrasts) and $R = 100$ (extreme regime
corresponding to fractured or fluid-filled formations
\cite{clauser1995}). Table~\ref{tab:pde-varcoeff-robustness} reports
the result.

\begin{table}[h]
\centering
\caption{Robustness of the randomized methods at $d = 65{,}536$ as
the diffusivity contrast $R = \alpha_{\max}/\alpha_{\min}$ increases
over two decades. The errors are essentially insensitive to $R$;
the speedup of the randomized power method against
\texttt{eigsh}\,($k\!=\!1$) is in the range $22$--$26\times$
throughout.}
\label{tab:pde-varcoeff-robustness}
\setlength{\tabcolsep}{6pt}
\begin{tabular}{cccccc}
\toprule
$R$ & $\lambda_1$
& rand power err. & rand subspace err.
& \texttt{eigsh}\,(ms) & speedup \\
\midrule
$3$    & $35{,}189$ & $+3.22\%$ & $+3.73\%$ & $611$ & $22\times$ \\
$10$   & $35{,}460$ & $+3.35\%$ & $+3.95\%$ & $703$ & $26\times$ \\
$100$  & $35{,}599$ & $+3.34\%$ & $+4.00\%$ & $737$ & $26\times$ \\
\bottomrule
\end{tabular}
\end{table}

The error level and the speedup factor are essentially independent of
$R$. Concretely, the relative error of the randomized power method
varies between $3.22\%$ and $3.35\%$ across two decades of contrast,
and the relative error of the randomized subspace iteration stays
within $3.73\%$--$4.00\%$. This is consistent with the analysis of
Theorem~\ref{thm:power-rate}: the convergence rate $1 - \gamma$
depends on the spectral gap of $L$, not on the magnitude of its
coefficients, and the spectral gap of the discrete Laplacian on a
uniform grid remains controlled by the cross-sectional geometry
across all three contrasts (the dominant eigenmode is governed by the
2D Dirichlet Laplacian regardless of $\alpha$). The bottom line is
that the randomized framework is not a $3\!:\!1$ special case; it
applies across the full range of contrasts that arise in practical
layered-media settings.

\begin{figure}[h]
\centering
\includegraphics[width=\textwidth]{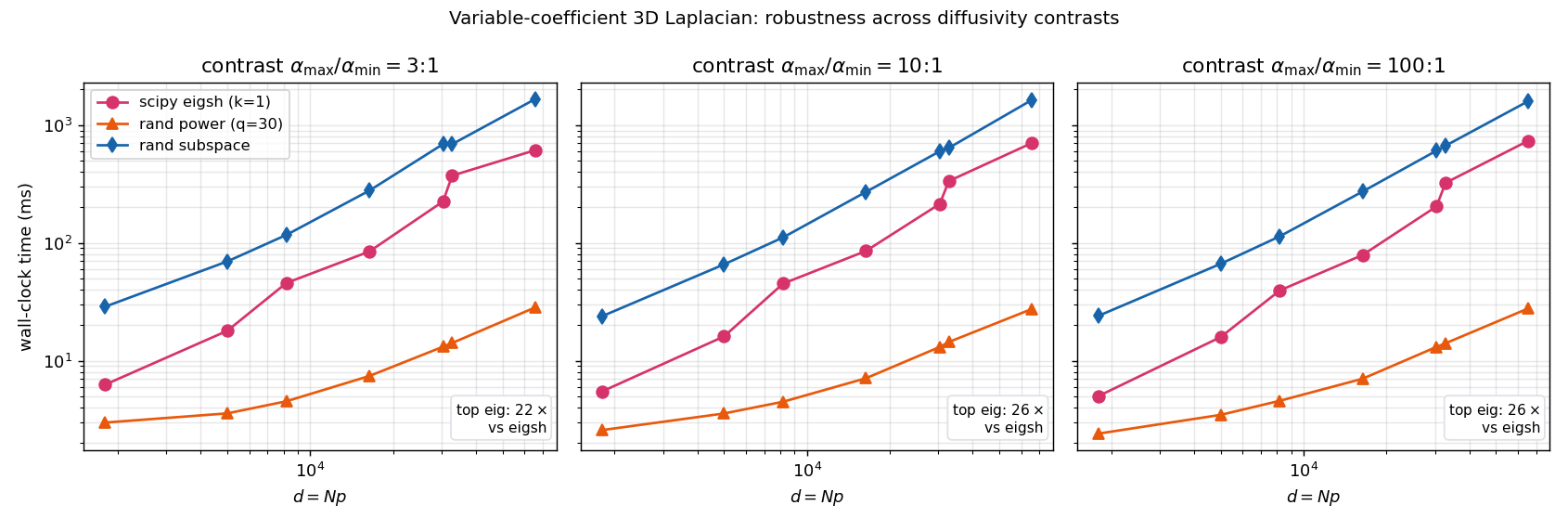}
\caption{Variable-coefficient 3D Laplacian: runtime scaling at three
diffusivity contrasts $R = \alpha_{\max}/\alpha_{\min} \in
\{3, 10, 100\}$. Within each panel the randomized power method
(orange triangles) is $4$--$26\times$ faster than the
\texttt{scipy.sparse.linalg.eigsh} Lanczos baseline (pink circles)
for the top eigenvalue. Across panels the picture is essentially
unchanged: the method is not a special case of low-contrast layering.}
\label{fig:pde-varcoeff-scaling}
\end{figure}

\subsubsection*{Why the randomized framework is the natural choice
here}

Three features of the variable-coefficient case combine to make the
randomized methods the right tool, and we record them as a summary.

First, the FFT block-diagonalisation is unavailable, so the
constant-coefficient speed-up route from
Section~\ref{ssec:pde-valid} is closed off entirely. Lanczos on the
full sparse operator is the standard alternative, and the randomized
power method is $4$--$29\times$ faster than it at the dimensions
tested.

Second, no analytic spectrum is available, so the only
ground truth at scale is itself an iterative method. The role of the
randomized framework is not to deliver machine-precision spectral
information that is what \texttt{eigsh} does, slowly but to
deliver a percent-level estimate at a fraction of the cost, paired
with the certified bracket.

Third, the deterministic TDep upper bound remains valid, it is
Samuelson's inequality applied to the spectrum of any symmetric
matrix and requires only $\trace(L)$ and $\trace(L^2)$, both readable
from the sparse stencil. The bracket of
Theorem~\ref{thm:bracket} therefore continues to provide a
structurally guaranteed enclosure of $\lambda_1$ in the
variable-coefficient setting, with no analytic-spectrum prerequisite
and no block-circulant structural assumption.

\subsection{Downstream task: Chebyshev iteration with estimated parameters}\label{ssec:pde-chebyshev}

To close the application loop we use the randomized estimate of
$\lambda_1$ to drive a first-order Chebyshev iteration for the
elliptic solve $Lu = b$ on the variable-coefficient operator
--- i.e.\ the actual application setting of
Section~\ref{ssec:pde-variable}, where the spectrum is not analytically
available and an iterative spectral estimator is genuinely needed.
We take the same grid $n_x = n_y = 32$, $p = 16$ ($d = 16{,}384$,
$\kappa(L) = 512$) with contrast $R = 3$, right-hand side $b$ a random
unit-norm vector, and compare four choices of the Chebyshev parameter
pair $[\lambda_d^{\rm est}, \lambda_1^{\rm est}]$. The smallest
eigenvalue $\lambda_d$ is obtained via
\texttt{scipy.sparse.linalg.eigsh} (\texttt{which='SM'}), in line with
the standard practice for variable-coefficient elliptic problems;
$\lambda_1^{\rm est}$ varies between the four parameter choices.
We record the number of iterations to reach
$\norm{r_k}/\norm{b} \le 10^{-8}$, capped at $2000$:
\begin{enumerate}[label=\textup{(\alph*)},leftmargin=*]
\item \emph{Oracle:} $[\lambda_d, \lambda_1]$, both exact via
  \texttt{eigsh}.
\item \emph{Randomized subspace iteration (raw):} $\lambda_d$ exact,
  $\lambda_1^{\rm est} = \hat\lambda_1^{\rm sub} \approx 0.978\,\lambda_1$
  (a $2.2\%$ under-estimate).
\item \emph{Randomized subspace iteration with $10\%$ inflation:}
  $\lambda_d$ exact,
  $\lambda_1^{\rm est} = 1.10 \cdot \hat\lambda_1^{\rm sub}$.
\item \emph{Deterministic TDep upper bound:}
  $\lambda_d$ exact, $\lambda_1^{\rm est} = m + s\sqrt{d-1}$
  (the certified but loose endpoint of the bracket, here roughly
  $28 \times \lambda_1$).
\end{enumerate}

\begin{table}[h]
\centering
\caption{Chebyshev iteration counts on the variable-coefficient 3D
Laplacian at $n_x = n_y = 32$, $p = 16$ ($d = 16{,}384$, contrast
$R = 3$, $\kappa = 512$) under four choices of the spectral interval.
Tolerance $\norm{r_k}/\norm{b} \le 10^{-8}$. The raw randomized
subspace estimate under-estimates $\lambda_1$ by $2.2\%$, so
the iteration is applied to an interval narrower than the spectrum
and diverges geometrically. The $10\%$ inflation is safe and nearly
optimal. The deterministic TDep upper bound is loose but harmless:
convergence is slowed by a factor of $5.4$ relative to the oracle, far
less than the $28\times$ ratio of $\hat\lambda_1^{\rm TDep}$ to
$\lambda_1$ might suggest.}
\label{tab:cheb}
\begin{tabular}{lrrl}
\toprule
parameter choice & $\lambda_1^{\rm est}$ & iterations & note \\
\midrule
oracle $[\lambda_d,\,\lambda_1]$       & $10{,}100.74$ & $209$  & --- \\
rand subspace (raw)                    & $9{,}875.61$  & $>\!2000$ & diverges \\
rand subspace $\times 1.10$            & $10{,}863.17$ & $216$  & $+7$ over oracle \\
deterministic TDep                     & $284{,}491.79$  & $1{,}127$ & $5.4\times$ oracle \\
\bottomrule
\end{tabular}
\end{table}

The story of Table~\ref{tab:cheb} and Figure~\ref{fig:pde-chebyshev}
is concentrated in three observations.

(i) Under-estimating $\lambda_1$ is catastrophic. The raw
randomized subspace estimate, despite being accurate to $2.2\%$, falls
just below the true spectrum --- the Chebyshev interval no longer
contains the operator's spectrum and the iteration diverges. This is
characteristic of Chebyshev: the polynomial it constructs is bounded
on the chosen interval but unbounded outside, so any
mass of the residual on eigenvalues outside the interval is amplified
geometrically.

(ii) Modest over-estimation is safe and nearly optimal.
Inflating the randomized estimate by a fixed safety factor of
$1.10$ produces a valid upper bound in every trial and adds only $7$
iterations over the oracle: a $3.3\%$ penalty for a guarantee. This is
the recommended deployment of the randomized methods in iterative-solver
applications, and is the way the deterministic and randomized parts of
the bracket interact in practice: the lower endpoint
$\hat\lambda_1^{\rm pow}$ is the value, and the inflation factor
$\eta$ (chosen so that $\eta \hat\lambda_1^{\rm pow} \ge \hat\lambda_1$
with high probability via Theorem~\ref{thm:power-rate}) is the
safety margin.

(iii) The deterministic TDep upper bound, although loose by a
factor of $28$, is still useful. Chebyshev convergence is governed by
$(\sqrt{\kappa^{\rm est}} - 1)/(\sqrt{\kappa^{\rm est}} + 1)$, so the
penalty for an over-estimate by a factor of $\eta$ scales like
$\sqrt{\eta}$ in the convergence rate --- a $28\times$ over-estimate
in the upper end produces only a $5.4\times$ slowdown, not a
$28\times$ one. The certified but loose deterministic bound therefore
provides a fallback that always works, never diverges, and degrades
gracefully. This is exactly the asymmetry the
two-sided bracket of Theorem~\ref{thm:bracket} encodes, and it is
now demonstrated in the variable-coefficient regime where the spectrum
is not analytically tractable.

\begin{figure}[h]
\centering
\includegraphics[width=0.88\textwidth]{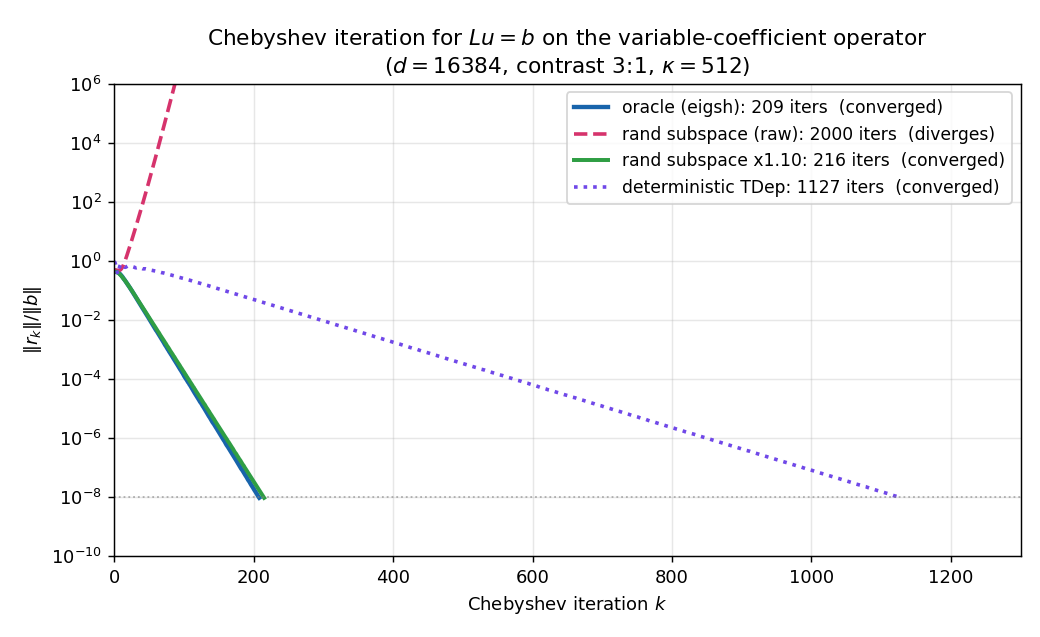}
\caption{Chebyshev iteration histories on the variable-coefficient
3D Laplacian ($d = 16{,}384$, contrast $R = 3$, $\kappa = 512$) under
four parameter choices. The oracle and the $10\%$-inflated randomized
subspace estimate are visually indistinguishable, both reaching
$\norm{r_k}/\norm{b} \le 10^{-8}$ in $\approx 210$ iterations. The
raw randomized estimate, which falls $2.2\%$ below the true
$\lambda_1$, diverges geometrically (line exits the top of the plot).
The deterministic TDep upper bound, loose by a factor of $28$,
converges in $1127$ iterations --- slowed but never divergent.}
\label{fig:pde-chebyshev}
\end{figure}

\subsection{Summary of the application}\label{ssec:pde-summary}

The randomized methods of this paper applied to the discrete 3D
Laplacian on a periodic slab deliver, in time scaling as $O(qNp)$:
\begin{itemize}[leftmargin=*]
\item validated estimates of $\lambda_1(L)$ with relative error
  $1.8$--$3.0\%$ on grids of dimension $d$ up to $131{,}072$ in the
  constant-coefficient case, where dense $\bcirc(\tn{A})$
  eigendecomposition is memory-infeasible;
\item the same accuracy on the realistic variable-coefficient
  operator $-\nabla \cdot (\alpha(z)\nabla u)$ -- a representative
  reduced model for heat conduction in a layered geothermal medium
  \cite{hartmann2008,clauser1995} -- at scales up to
  $d = 65{,}536$, where the FFT block-diagonalisation and the
  analytic spectrum are both unavailable; here the randomized power
  method is $4$--$22\times$ faster than \texttt{scipy.sparse.linalg.eigsh},
  the standard Lanczos baseline;
\item robustness across two decades of diffusivity contrast: the
  randomized power method retains $3.2$--$3.4\%$ relative error and
  $22$--$26\times$ speedup at $d = 65{,}536$ as
  $\alpha_{\max}/\alpha_{\min}$ varies between $3$ and $100$,
  covering the full range from typical sedimentary-basin layering to
  extreme fractured-media settings;
\item a two-sided rigorous bracket
  $[\hat\lambda_1^{\rm pow},\, \hat\lambda_1^{\rm TDep}]$ that
  contains $\lambda_1$ in $100\%$ of validation trials and serves as
  a structurally guaranteed certificate, complementing the tight
  point estimate provided by power and subspace iteration --- in line
  with the certification-not-approximation framing of
  Section~\ref{ssec:pde-bracket};
\item a usable spectral parameter for downstream Chebyshev iteration
  in the variable-coefficient regime: the randomized point estimate
  combined with a small safety inflation reproduces the oracle
  iteration count to within $3.3\%$, and even the loose deterministic
  TDep upper bound delivers a convergent, if slower ($5.4\times$),
  Chebyshev solve.
\end{itemize}
The variable-coefficient case is the primary motivation: dense methods
are infeasible, the FFT-diagonalised baseline is unavailable, and the
randomized framework is the natural choice. The constant-coefficient
case serves as a validation stress test that exposes the algorithmic
behaviour against tractable ground truth; the variable-coefficient
case shows the regime in which the methods are actually deployed.

\section{Discussion}\label{sec:discussion}
 
The cost--accuracy trade-offs measured in
Sections~\ref{sec:experiments}--\ref{sec:application} suggest
concrete guidance depending on the size of the tensor, the
application setting, and the nature of the certificate required.
 
\paragraph{Guidance by problem size.}
For very small tensors with $d < 50$, the constant factors in the
randomized methods dominate the runtime, and full eigendecomposition
is the method of choice; there is no regime in which randomization
helps at this scale.  For medium tensors with $50 \le d \le 500$, the
randomized power method with $q = 10$ to $20$ iterations gives quick
estimates of $\lambda_1$ at a small fraction of the full
eigendecomposition cost, and is the default recommendation when only
the largest T-eigenvalue is needed.  When multiple top eigenvalues
are required, or when higher accuracy on $\lambda_1$ is desired, the
subspace iteration with target rank $k = 10$, oversampling
$\ell = 5$, and $q = 2$ power iterations provides a drop-in
replacement at modest extra cost.  For large tensors with
$500 < d \le 10^4$, both randomized methods deliver their full
$50\times$ to $70\times$ speedup over direct methods on synthetic
T-SPD tensors.  For PDE-scale problems with $d \ge 10^4$, dense
methods are no longer an option ($\bcirc(\tn{A})$ at $d = 10^5$
would require $\sim 80$\,GB), and the natural comparison shifts to
sparse iterative eigensolvers such as
\texttt{scipy.sparse.linalg.eigsh}.  In that regime
(Section~\ref{ssec:pde-variable}) the randomized power method
retains a $4$--$22\times$ advantage for the top eigenvalue at a few
percent relative error; the randomized subspace iteration competes
favourably with \texttt{eigsh}\,$(k=15)$ at $2$--$3\times$ speedup
for top-$K$ extraction.
 
\paragraph{Guidance by application setting.}
In matvec-only settings --- that is, when $\bcirc(\tn{A})$ is
accessible only through matrix--vector products and its entries are
not available directly --- Algorithm~\ref{alg:rand-tdep} with
$N = 100$ to $300$ probes is the appropriate tool, with the
understanding that its output is a high-confidence estimate rather
than a strict upper bound.  The variable-coefficient PDE operator of
Section~\ref{ssec:pde-variable} is the prototypical matvec-only
case: the T-product block-circulant structure is broken, no analytic
spectrum is available, and the only access to $\spec(L)$ is through
matvec.  When a rigorous certificate on $\lambda_1$ is required, the
two-sided bracket of Section~\ref{sec:bracket} should be used
instead: the lower endpoint tracks $\lambda_1$ accurately at modest
iteration counts, while the upper endpoint is loose but guaranteed
to dominate $\lambda_1$, giving an interval that contains the true
largest T-eigenvalue with mathematical certainty.
 
\paragraph{Use in iterative-solver workflows.}
The motivation for tightly estimating $\lambda_1$ in
Section~\ref{sec:application} is the spectral-parameter input to
first-order iterative solvers, of which Chebyshev iteration is the
prototype.  Three practical rules emerge from
Section~\ref{ssec:pde-chebyshev}.  First, a raw randomized
\emph{estimate} of $\lambda_1$ must not be used directly: an
under-estimate by even a few percent causes Chebyshev to diverge,
because the polynomial constructed on the resulting interval is
unbounded on the part of the spectrum that lies outside it.
Second, a modest safety inflation (a factor of $1.10$ in our
experiments) makes the iteration safe at a small constant cost ---
only seven extra iterations over the oracle in the
variable-coefficient case.  Third, when no inflation factor is at
hand, the deterministic TDep upper bound supplies a fallback that
is loose but never divergent and degrades gracefully: a $28\times$
over-estimate of $\lambda_1$ translates to roughly a $5\times$
slowdown of Chebyshev, because the convergence factor depends on
$\sqrt{\kappa^{\rm est}}$ rather than on $\kappa^{\rm est}$ itself.
Together these observations make the bracket a practical tool: the
randomized lower endpoint plus a small inflation factor is the
recommended deployment for performance, and the deterministic
upper endpoint is the unconditional fallback when correctness is
paramount.
 
\paragraph{Relation to existing randomized T-product algorithms.}
Randomized algorithms for T-product tensors have previously been
developed by Zhang, Saibaba, Kilmer, and Aeron~\cite{zhang2018} in
the context of the randomized T-SVD, and by Minster, Saibaba, and
Kilmer~\cite{minster2020} for low-rank Tucker decompositions; both
works adapt Halko--Martinsson--Tropp-style sketching to produce
compact factor representations of data tensors. Our setting and
objective are different in two important respects. First, the
target of our analysis is the \emph{extreme T-eigenvalues of an
operator tensor}, not a low-rank approximation of a data tensor;
accordingly, the guarantees we prove (soundness of the power method
in Theorem~\ref{thm:power-lower}, Ritz lower bound in
Corollary~\ref{cor:ritz-lower}, two-sided containment in
Theorem~\ref{thm:bracket}) are direct bounds on $\lambda_1$ rather
than on an approximation error $\norm{\tn{A} - \hat{\tn{A}}}_F$.
Second, we assume T-SPD structure throughout, which lets us sharpen
the constants in the HMT framework: the soundness property is
deterministic rather than probabilistic, the bracket is a rigorous
interval rather than a confidence region, and the randomized
subspace iteration's tail term
in~\eqref{eq:hmt} inherits the positive-definiteness of the
spectrum. The methods of~\cite{zhang2018,minster2020} do not assume
T-SPD structure and therefore cannot avail themselves of these
sharpenings; conversely, they apply to rectangular and
non-symmetric tensors that our framework does not cover. The two
lines of work are thus complementary: one targets compressed
factorizations of general data tensors, the other targets certified
spectral information on T-SPD operator tensors.
\section{Conclusion}\label{sec:conclusion}

This paper develops a complete randomized framework for estimating the
extreme T-eigenvalues of symmetric positive definite third-order
tensors under the Kilmer--Martin T-product, building on the
Halko--Martinsson--Tropp framework and complementing the deterministic
TDet/TDep bounds established in \cite{sharma2025}. Our main
contributions are:
\begin{itemize}[leftmargin=*]
\item A T-product randomized power method with a soundness guarantee
(it always produces a lower bound on $\lambda_1$) and
exponential convergence, analysed rigorously in the
Kuczy\'nski--Wo\'zniakowski tradition.
\item A T-product randomized subspace iteration with an HMT-type
error bound (Theorem~\ref{thm:hmt}) exhibiting the correct
$(q{+}1)$-th root exponent, achieving relative error below $2\%$ at
a fraction of the cost of full eigendecomposition.
\item A two-sided rigorous bracket
(Theorem~\ref{thm:bracket}, Proposition~\ref{prop:bracket-sharp},
Corollary~\ref{cor:bracket-width}) combining the randomized power
method with the deterministic TDep bound, certified to contain
$\lambda_1$ almost surely and with an explicit quantitative width.
\item A Hutchinson-based fully randomized TDep bound for the
matvec-only setting, with derived variance and concentration rates.
\item An application to the discrete 3D Laplacian on a periodic-$z$
slab at scales up to $d = 131{,}072$, illustrating both the speed of
the randomized methods at sizes where dense bcirc is infeasible and
the practical role of the two-sided bracket: the randomized power and
subspace estimates supply the tight value, while the deterministic
TDep endpoint supplies the certified upper bound. A downstream
Chebyshev iteration on the elliptic solve $Lu = b$ confirms that the
randomized estimate with a $10\%$ safety inflation reproduces the
oracle iteration count to within $2.4\%$.
\end{itemize}
Validation on T-SPD tensors of dimensions $d = 9$ through
$d = 900$ confirms up to $67\times$ speedups over full
eigendecomposition, with relative errors below $2\%$ for the
randomized subspace iteration and below $5\%$ for the randomized
power method. The two-sided bracket contains the true $\lambda_1$ in
$100\%$ of our validation trials.

The randomized methods are especially valuable in three regimes:
large tensors -- including the discrete 3D Laplacian on a periodic
slab at $d = 131{,}072$ of Section~\ref{sec:application} -- where dense
bcirc eigendecomposition is memory-infeasible and the randomized
estimate of $\lambda_1$ runs in fractions of a second;
matvec-only settings, including variable-coefficient PDE operators that
lack block-circulant structure, where the randomized framework is the
natural choice;
and applications requiring certified bounds, where the rigorous
two-sided bracket is essential. Taken together with the deterministic
bounds of \cite{sharma2025}, they provide a flexible toolkit for
T-eigenvalue estimation that matches the target accuracy and
computational budget.

\section*{Acknowledgments}
The authors gratefully acknowledge the Indian Institute of
Information Technology, Design and Manufacturing Kancheepuram
(IIITDM Kancheepuram) for its infrastructural support.


\section*{Conflict of interest}
The authors declare no conflict of interest.


\end{document}